\documentclass[reqno,10pt]{amsart}
\usepackage{amssymb,amsmath,amsthm,amsfonts,color,graphicx}
\usepackage{mathrsfs,dsfont,a4wide,esint}

\theoremstyle{plain}
\newtheorem{theorem}{Theorem}[section]
\newtheorem{proposition}[theorem]{Proposition}

\theoremstyle{definition}
\newtheorem{remark}[theorem]{Remark}

\numberwithin{equation}{section}

\newcommand{\expec}[1]{\mathbb{E}[#1]}
\newcommand{\Expec}[1]{\mathbb{E}\Big[#1\Big]}

\newcommand{\Id}{\rm{\sf i }}
\newcommand{\R}{{\mathbb R}}
\newcommand{\Z}{{\mathbb Z}}
\newcommand{\N}{{\mathbb N}}
\newcommand{\loc}{{\mathrm{loc} }}
\newcommand{\uloc}{{\mathrm{uloc} }}

\newcommand{\Pc}{{\mathcal P}}
\newcommand{\calQ}{{\mathcal Q}}

\newcommand{\B}{{\mathcal B}}

\newcommand\T{{\mathscr T}}
\renewcommand\S{{\mathscr S}}
\newcommand\M{{\mathscr M}}
\newcommand\I{{\mathscr I}}

\newcommand\ds{\displaystyle}

\newcommand{\Msym}{{\bf M}^{N}_{\rm sym}}

\newcommand\Mb{\mathcal M_b}

\newcommand{\Om}{\Omega}
\newcommand{\Omb}{\overline{\Omega}}

\newcommand{\weak}{\rightharpoonup}

\newcommand{\e}{\varepsilon}
\newcommand\ed{\e^\delta}
\newcommand{\ph}{\varphi}
\newcommand\phd{\ph^\delta}
\newcommand\td{\tau^\delta}
\newcommand\qd{q^\delta}
\newcommand\wdj{w_j^\delta}
\newcommand\xd{(\frac x\delta)}
\newcommand\sd{\sigma^\delta}
\newcommand\thd{\theta^\delta}
\newcommand\zd{\zeta^\delta}
\newcommand\Ld{L^\delta}
\newcommand\ud{u^\delta}
\newcommand\tud{\tilde u^\delta}
\newcommand\Md{M^\delta}
\newcommand\Wdj{W_j^\delta}
\newcommand\teh{\tilde\e^h}

\newcommand\diam{\mathrm{diam}\,}

\newcommand{\be}[1]{\begin{equation}\label{#1}}
\newcommand{\ee}{\end{equation}}

\newcommand{\ba}{\left\{\begin{array}}
\newcommand{\ea}{\end{array}\right.}

\newcommand{\de}[2]{\frac{\partial #1}{\partial #2}}

\newcommand{\dvg}{\operatorname{div}}

\title [Enhancement in homogenization of elasto-dielectrics]{Enhancement of  elasto-dielectrics
by homogenization of active charges}
\author[G.A. Francfort] {Gilles A. Francfort}
\address[Gilles Francfort]{LAGA, Universit\'e Paris-Nord,
Avenue J.-B. Cl\'ement 93430 - Villetaneuse, France}
\email[G. Francfort]{gilles.francfort@univ-paris13.fr}
\author[A. Gloria] {Antoine Gloria}
\address[Antoine Gloria]{Sorbonne Universit\'e, CNRS, Universit\'e de Paris, Laboratoire Jacques-Louis Lions, 75005 Paris, France \& Institut Universitaire de France (IUF) \& Universit\'e Libre de Bruxelles, D\'epartement de Math\'ematique, 1050 Brussels, Belgium}
\email[A. Gloria]{gloria@ljll.math.upmc.fr}
\author[O. Lopez-Pamies] {Oscar Lopez-Pamies}
\address[Oscar Lopez-Pamies]{Department of Civil and Environmental Engineering,  University of Illinois at Urbana--Champaign, IL 61801, USA}
\email[O. Lopez-Pamies]{pamies@illinois.edu}

\begin{document}
\vskip .2truecm
\begin{abstract}{\vskip .3truecm  We investigate the PDE system resulting from  even electromechanical coupling in elastomers. Assuming a periodic microstructure and a periodic distribution of micro-charges of a prescribed order, we derive the homogenized system. The results depend crucially on periodicity (or adequate randomness) and on the type of microstructure under consideration. {We also  show electric
enhancement if the charges are carefully tailored to the homogenized electric field and explicit that enhancement, as well as the corresponding  electrostrictive enhancement in a dilute regime.}

}
\end{abstract}

\maketitle
\centerline{\scriptsize  Version of \today}

\section{Introduction}

Ever since the discovery of the piezoelectric behavior of several types of minerals --- including quartz, tourmaline, and Rochelle salt --- by Pierre and Jacques Curie in the 1880s \cite{Curie1880,Curie1881}, deformable dielectrics have been an object  of uninterrupted interest in  fields ranging from materials science to mathematics. This has been reinforced since the turn of the millennium when soft organic dielectrics were ``re-discovered'' as a class of materials with high technological potential.

In contrast to the {odd} coupling between  mechanical and electric field, a characteristic of the hard deformable dielectrics investigated by the Curie brothers, soft organic dielectrics typically exhibit {even} electromechanical coupling. From a mathematical point of view, this means that the governing equations involved   exhibit nonlinearity, even in the simplest asymptotic setting of small deformations. Furthermore,   space charges  varying at the length scale of the microstructure may assert their presence, as is the case, for example, in porous polymer electrets \cite{Bauer,Hillenbrand} and polymer nano-particulate composites \cite{Huang,Nelson}. This  translates into  equations that contain a rapidly oscillating source term and leads  to anomalous behaviors \cite{GGLP,LGMO}.

Our goal in this study is to investigate the homogenization of elasto-dielectrics with even electromechanical coupling that contain space charges that vary at the length scale of their microstructure; a  formal analysis of that problem was presented in \cite{LLP}. In addition to ignoring dissipative effects, we restrict attention to materials with periodic (or with adequate random) microstructure, quasi-static electromechanical loading conditions, and further focus on the asymptotic setting of small deformations and moderate electric fields. The derivation of the relevant local governing equations goes as follows.

Consider an elastic dielectric that occupies a bounded domain $\Om\subset\mathds{R}^N$ with boundary $\partial\Om$ in its undeformed, stress-free, and polarization-free ground state. Material points are identified by their initial position vector $x$ in $\Om$ relative to some fixed point. Upon  application of mechanical loads and electric fields, the position vector $x$ of a material point moves to a new position  $v(x) = x + u(x)$, where $u$ denotes the displacement field. The associated deformation gradient is denoted by $F(x)= I + \nabla u(x)$. In the absence of magnetic fields, free currents, and body forces, and with no time dependence (see, e.g., \cite{Ogden}), Maxwell's and the momentum balance equations require that
$$\begin{cases}
{\rm div}  D=Q,\; {\rm curl}\, E=0, \;x\in \mathds{R}^N\\[2mm]
{\rm div}\, S=0,\; S F^T=F S^T,\; x\in \Om,
\end{cases}$$
where $D(x)$, $E(x)$, $S(x)$ stand for the Lagrangian electric displacement field, the Lagrangian electric field, and the ``total'' first Piola-Kirchhoff stress tensor, while
$Q(x)$ stands for the density (per unit undeformed volume) of space charges. Further,
\begin{equation*}
D(x)=-\dfrac{\partial W}{\partial E}(x,F(x),E(x))\qquad {\rm and}\qquad S(x)=\dfrac{\partial W}{\partial F}(x,F(x),E(x)),
\end{equation*}
where the ``total'' free energy $W(x,F,E)$ is an objective function of the deformation gradient tensor $F$ and an {\it even} and objective function of the electric field $E$, namely, $W(x,F,E)=W(x,QF,E)=W(x,F,-E)$ for all $Q\in SO(N)$ and arbitrary $F$ and $E$. The objectivity of $W$ implies that the balance of angular momentum $S F^T=F S^T$ is  automatically satisfied. Faraday's law ${\rm curl}\, E=0$ can also be satisfied automatically by the introduction of an electric potential $\varphi(x)$ such that $E(x)=-\nabla\varphi(x)$. Thus, only Gauss's law ${\rm div} \, D=Q$ and the balance of linear momentum ${\rm div}\, S=0$ remain.

Now, setting $H:=F-I$, a Taylor expansion of $W$ about the ground state $F=I$,  $E=0$ yields
\begin{equation*}
W(x,F,E)=-\dfrac{1}{2}E\cdot\varepsilon(x)E+\dfrac{1}{2}H\cdot L(x)H+H\cdot (M(x)(E\otimes E))-E\otimes E\cdot(\mathcal{T}(x)(E\otimes E))+\ldots,
\end{equation*}
where $\varepsilon(x):=-\partial^2W(x,I,0)/\partial E^2$ is the permittivity tensor, $L(x):=\partial^2W(x,I,0)/\partial F^2$ is the elasticity tensor, $M(x):=1/2\partial^3W(x,I,0)/\partial F\partial E^2$ is the electrostriction tensor, and $\mathcal{T}(x):=-1/24$ $\partial^4W(x,I,0)/\partial E^4$ is the permittivity tensor of second order. It follows that the constitutive relations that describe the electromechanical response of the elastic dielectric specialize to
\begin{eqnarray*}
D(x)=\varepsilon(x) E(x)+{H(x)\cdot M(x)E(x)}+ \mathcal{T}(x)(E(x)\otimes E(x)\otimes E(x))+\ldots\\[2mm] S(x)=L(x)H(x)+M(x)(E(x)\otimes E(x))+\ldots
\end{eqnarray*}
Taking the magnitude of the deformation measure $H$ to be of order $\zeta$, with $0<\zeta<\!<1$, it follows in turn that the electric field $E$ must be of order $\zeta^{1/2}$ if  the elastic dielectric  is to display electromechanical coupling around its ground state. To leading order, we then get
\begin{equation*}
D(x)=\varepsilon(x) E(x)\qquad {\rm and}\qquad S(x)=L(x)H(x)+M(x)(E(x)\otimes E(x)).
\end{equation*}
This is the so-called scaling of small deformations and moderate electric fields; within this scaling, by the same token, the space charge density $Q$ must be of order $\zeta^{1/2}$.

We  now detail the governing equations for the problem under investigation in this work. Assuming  periodicity of the microstructure, the permittivity, elasticity, and electrostriction tensors ($\e(y),L(y),M(y),$ respectively) that characterize the local elastic dielectric response of the material are defined on a unit cell (or, more precisely, on a unit torus $\T$)  and they are periodically rescaled by a small parameter $\delta$ to reflect the size of the microstructure. The resulting tensors are respectively denoted by $\ed(x),\Ld(x),\Md(x).$

\begin{figure}[h!]
\includegraphics[scale=.5]{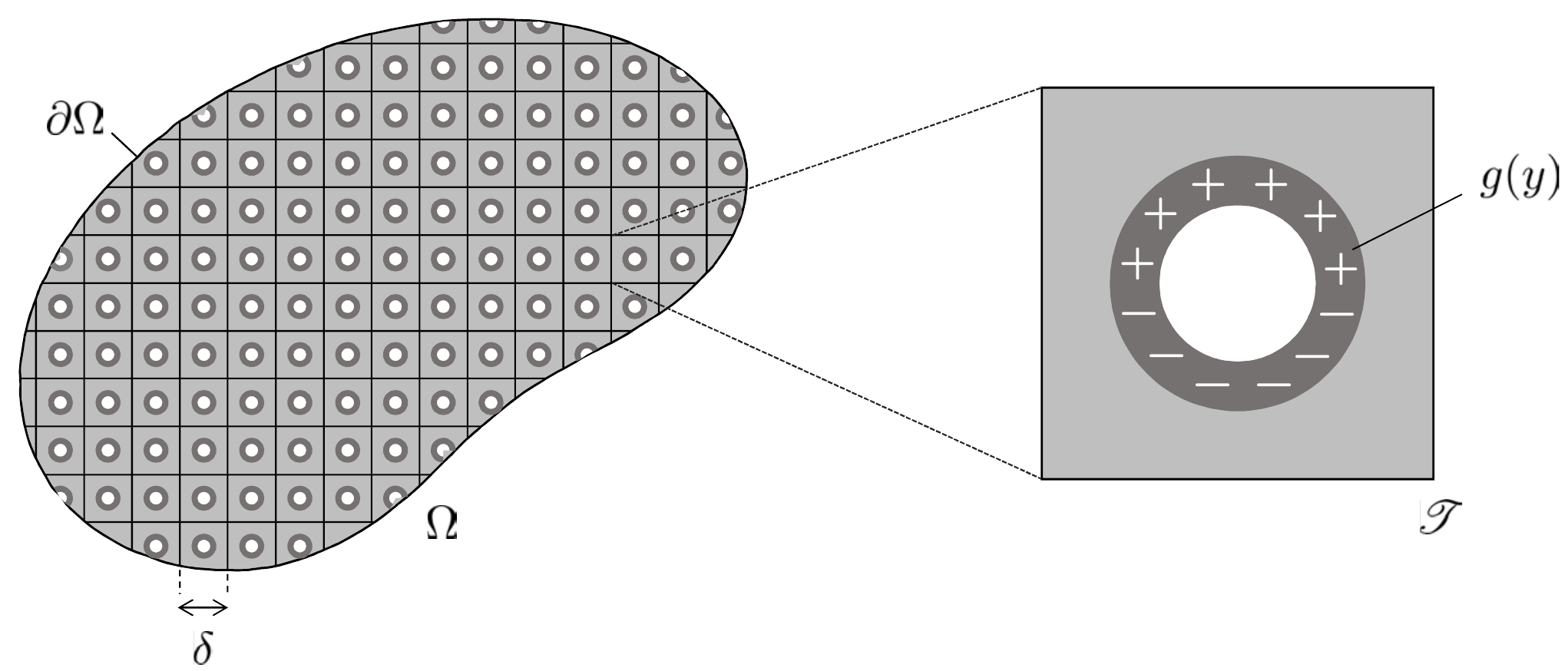}
\caption{\scriptsize {(a) The elastic dielectric composite. (b) The unit cell  with  the space-charge density.}}
\label{fig1}
\end{figure}
Moreover, the material is assumed to contain a distribution of periodically distributed space charges with density $g(y)$ such that
\be{eq.cn}\int_\T g(y)\ dy=0\ee
so as to preserve local charge neutrality, rescaled in a manner similar to that of the microstructure and modulated by a slowly varying macroscopic charge $f(x)$. These space charges can be passive or active. In the case of passive charges, the slowly varying macroscopic charge $f(x)$ is fixed from the outset. Physically, this corresponds to materials wherein space charges are ``glued'' to material points and remain so regardless of the applied mechanical loads and electric fields. This is the case, for instance, of porous polymer electrets for which the space charges are fixed at the walls of the pores. In the case of active charges,  the slowly varying macroscopic charge $f(x)$ is identified as the resulting macroscopic field for the electric potential and hence depends on the applied electric field. Physically, this corresponds to materials wherein space charge are locally mobile. This is the case, for instance, of polymer nano-particulate composites for which the space charges are locally mobile around the interfaces between the polymer and the nano-particles. Figure \ref{fig1} illustrates a schematic of the material and of its periodic microstructure and space charge content.

The relevant governing equations are
\be{eq.system}
\ba{l} \dvg  \ed \nabla \phd = \frac 1\delta g^\delta f\\[2mm]\dvg \left[\Ld \nabla\ud+\Md (\nabla \phd\otimes  \nabla \phd)\right]=0\ea\ee
for the electric potential $\phd$ and the displacement field $\ud$.   For simplicity , the boundary conditions are taken to be of Dirichlet type, that is,
\[\phd=\Phi\quad, \ud=0 \quad\mbox{on }\partial\Om.\]
Note that imposing Dirichlet boundary conditions on the electric potential amounts to  considering Gauss' law  inside the domain $\Om$, and not in $\mathds{R}^N$, a situation which corresponds to electrodes being placed along the entire boundary of $\Om$.

\begin{remark}\label{rem.order.delta} The heuristic justification of the presence of the term $1/\delta$ in front of the space charges $g^\delta f$ is as follows. Because of charge neutrality (see \eqref{eq.cn}), multiplication of the source term by $\delta^q$ with $q>-1$ would result in a homogenized dielectric equation  without any source term, that is an equation of the form
\[\dvg \e^h \nabla\ph=0 \ \mbox{ in }\Om,\]
where $\e^h$ is the homogenized permittivity tensor defined later in \eqref{eq.hom-diel}. Thus, the lowest $\delta$-order at which microscopically distributed charges will impact the homogenized dielectric equation is $\delta^{-1}$.  Of course, one can always add lower order source terms as emphasized  in Remark \ref{rem.source-term} below, but, their impact will disappear in the effective behavior unless charge neutrality is forsaken for those terms.
\hfill\P\end{remark}

The first objective of this work is to determine the purely dielectric macroscopic behavior of the material for an arbitrary but fixed ({\it i.e.}, passive) distribution of space charges in the limit when the period $\delta$ of the microstructure goes to $0$. This will be achieved in Section \ref{sec.ch}.

The second objective is to demonstrate  that  dielectric enhancement can always be achieved for the purely dielectric macroscopic behavior when adequate  active space charges are introduced.  To do that we need to identify $f(x)$ with the resulting macroscopic field for the electric potential.  We demonstrate that for a two-phase inclusion type microstructure, it is always possible to produce enhancement. Further, in the case of dilute inclusions, we propose an argument  inspired by the Clausius-Mossotti formula that yields an explicit value for that enhancement {and for  ``manufacturable" charges}.
This is the object of Section \ref{sec.ac} which we have placed at the end of this paper.

Finally, we determine the homogenized equations for the coupled elastic-dielectric behavior of the material. This is the object of Section \ref{sec.hed}. Doing so necessitates better convergence properties on the dielectric micro-macro analysis than those provided by Section \ref{sec.ch}.
To do so we combine large-scale regularity due to homogenization with local regularity properties
 that hold for two-phase microstructures with smooth inclusions. The technical details are the object of Section \ref{sec.ic}. In the last part of Section \ref{sec.ac}, we also investigate the elastic enhancement  for the dilute case already alluded to above.

Although the results and their proofs are written in the case of periodic media, they can all be extended to random media (with suitable mixing conditions), as we quickly argue in the appendix.

Notationwise, we denote by $\Msym$ the space of symmetric $N\times N$-matrices and by $\cdot$ the Euclidean inner product between vectors in $\R^N$ or the Fr\"obenius inner product between elements of $\Msym$, that is $e\cdot e'={\rm tr\;}ee'$ {with} $e,e'\in \Msym$. We will denote by $B_r(x)$ the open ball of center $x$ and radius $r$ and by $\rm{\sf i }$ the identity matrix.

We will sometimes identify the torus and its subsets with the unit cube  $Y=\Pi_{i=1,...,N}[0,1)\subset\R^N$  and the corresponding subsets (denoted with the corresponding roman character) through the canonical identification $\mathbf i$ between $\T$ and $\{z+Y:z\in\mathbb Z^N\}=\R^N$. Also we will adopt the following convention for a function $\zeta$ defined on $\T$. We will say that $\zeta\in H^1(\T)$ if, and only  $z=\zeta\circ\mathbf i$  is such that $z\in H^1_{\rm loc}(\R^N)$; note  that $z$ is $Y$-periodic. Further, if $\e \in L^\infty(\T;\Msym)$, we will write $\dvg{\e \nabla \zeta}$ for $\dvg{\{(\e\circ\mathbf i)\nabla z\}}$ and denote by $\zeta^\delta$ the periodic $H^1_{\rm loc}$-function $z(x/\delta)$ which we will also write as $\zeta(x/\delta)$. Similarly, we will denote
by $\int_\T (\nabla)\zeta(y)\ dy$ the integral $\int_Y (\nabla)z(y)\ dy$.

The rest of the notation is standard.

\section{Classical homogenization of the dielectrics}\label{sec.ch}
\label{sec:dielectrics}

In this section, we consider the dielectric part of our problem and propose to pass to the limit as the period goes to $0$. As already noted,  {structural assumptions such as }periodicity {(a random distribution with good enough mixing properties would do as well)}, while essential in the next section, {are not necessary assumptions} when handling the scalar dielectric equation; see Remark \ref{rem.Hconv} below.

So, on $\Om$, a bounded Lipschitz domain of $\R^N$, we consider the equation
\be{eq.diel}\ba{l} \dvg \ed \nabla \phd= \frac 1\delta g\xd f(x)\\[2mm]
\phd=\phi \mbox{ on }\partial\Om\ea\ee
with $f\in W^{1,\infty}(\Om)$, $g\in L^2(\T;\R^N)$ and $\int_\T g(y)\ dy=0$, $\phi\in H^{\frac 12}(\partial\Om)$ and $\ed(x):=\e\xd$ where $\e(y)\in L^\infty(\T;\Msym)$ with $\gamma |\xi|^2\le \e(y)\xi\cdot\xi\le\beta|\xi|^2$ for some $0<\gamma<\beta<\infty$.

We define $\psi$ to be the unique solution in $H^1(\T)$ of
\be{eq.psi}\ba{l} \triangle \psi(y)=  g(y)\\[2mm]
\int_\T \psi(y)\ dy=0\ea\ee
and note that, by elliptic regularity, $\psi \in H^2(\T)$. We set
\be{eq.tau}\tau(y):=\nabla \psi(y),\quad\td(x):= \nabla\psi\xd,\ee
so that \eqref{eq.diel} reads as
\be{eq.diel2}\ba{l} \dvg (\ed \nabla \phd- f\td)= -\td \cdot \nabla f\\[2mm]
\phd=\phi \mbox{ on }\partial\Om.\ea\ee
From \eqref{eq.diel2} and Poincar\'e's inequality, we immediately obtain that
\[\phd \mbox{ is bounded in } H^1(\Om) \mbox{ independently of } \delta\]
and, upon setting
\[\qd:= \ed\nabla\phd-f\td,\]
that
\[\qd \mbox{ is bounded in } L^2(\Om;\R^N) \mbox{ independently of } \delta.\]

Thus, up to a subsequence (not relabeled), we conclude that
\be{eq.wconv}\ba{ll} \phd \weak \ph &\mbox{ weakly in } H^1(\Om)\\[2mm]
\qd \weak q &\mbox{ weakly in } L^2(\Om;\R^N).\ea\ee
Of course,
\be{eq.div=0} \dvg q=  0\ee
since $\td \stackrel{L^2(\Om;\R^n)}{\weak} \fint_\T \nabla \psi(y)\ dy =0$.
It remains to identify $q$.

To that effect, consider the periodic corrector $w_j$ defined as follows. Set $\chi_j$ to be the unique solution in $H^1(\T)$ to
\be{eq.chij}\ba{l}\dvg \e \nabla (\chi_j+y_j)=0\\[2mm]\int_\T \chi_j\ dy=0.\ea\ee
Then
\[w_j:=\chi_j+y_j\]
Set
\be{def.wdj}\wdj(x):= \delta\chi_j\xd+x_j\ee
and note that $\nabla\wdj(x)=(\nabla w)\xd$. Then, on the one hand, the div-curl Lemma  \cite{murat} (or integration by parts) implies that, for any $\zeta\in C^\infty_c(\Om)$,
\[\int_\Om \zeta \qd\cdot \nabla\wdj\ dx \longrightarrow \int_\Om \zeta q\cdot\vec{e}_j\ dx=\int_\T \zeta q_j \ dy.\]

On the other hand define, according to classical elliptic homogenization \cite[Chapter 1]{bensoussan.lions.papanicolaou}, the symmetric constant matrix $\e^h$ as
\be{eq.hom-diel}\e^h\vec{e}_j:=\fint_\T \e(y)\nabla w_j \ dy.\ee
 Since $\e$ is symmetric, another application of the div-curl Lemma yields
\[\int_\Om \zeta \qd\cdot\nabla\wdj\ dx=\int_\Om \zeta \nabla\phd\cdot\ed \nabla\wdj\ dx - \int_\Om \zeta f\td\cdot\nabla\wdj\ dx\longrightarrow \int_\Om \zeta \nabla\ph\cdot\e^h\vec{e}_j\ dx- \int_\Om \zeta fa_j\ dx,\]
with
\be{eq.defaj}a_j:=\fint_\T \tau(y)\cdot\nabla w_j(y)\ dy=\fint_\T \nabla\psi(y)\cdot(\vec{e}_j+\nabla \chi_j(y))\ dy=-\fint_\T g(y)\chi_j(y)\ dy.
\ee
Hence,
\be{eq.q}
q=\e^h\nabla\ph-af
\ee
with $a\in\R^N$ given through \eqref{eq.defaj}.

A classical result of $H$-convergence  is that
\[\gamma |\xi|^2\le\e^h\xi\cdot\xi\le \beta|\xi|^2.\]
Thus, in view of \eqref{eq.div=0}, we conclude that $\ph$ is the unique $H^1(\Om)$-solution of
\be{eq.hom} \ba{l}\dvg \e^h\nabla\ph = a\cdot\nabla f\\[2mm]\ph=\phi \mbox{ on }\partial\Om,\ea\ee
so that the entire sequence $(\phd,\qd)$ converges to $(\ph, q)$ weakly in $H^1(\Om)\times L^2(\Om;\R^N)$.


\medskip

We now strive to improve the weak convergence results with the help of correctors.
To that effect we introduce $\theta\in H^1(\T)$ to be the unique solution to
\be{eq.eq-t}\begin{cases}\dvg (\e\nabla \theta -\tau)=0\; (\mbox{or still }\dvg \e\nabla \theta=g)\\[2mm]\int_\T\theta(y)\ dy=0.\end{cases}\ee
We set
\[\sigma:=\e\nabla \theta -\tau,\]
and
\be{def.thd}\thd(x):=\delta\theta\xd \quad \sd(x):=\sigma\xd, \ee
so that $\sd=\ed\nabla\thd-\td$
and note that,  by symmetry of $\e$ and \eqref{eq.chij}, since $\tau$ has zero average over $\T$ and in view of \eqref{eq.defaj},
\begin{multline}\label{eq.conv.sig}\sd_j \weak b_j:=\fint_\T\e(y)\nabla\theta(y)\cdot\nabla y_j\ dy=-\fint_\T\nabla\theta(y)\cdot \e(y)\nabla \chi_j(y)\ dy\\[2mm]=-\fint_\T(\e(y)\nabla\theta(y)-\tau(y))\cdot \nabla \chi_j(y)\ dy-\fint_\T\tau(y)\cdot \nabla \chi_j(y)\ dy\\[2mm]=-\fint_\T\tau(y)\cdot \nabla \chi_j(y)\ dy=-\fint_\T \tau(y)\cdot\nabla w_j(y)\ dy=-a_j, \;\mbox{ weakly in } L^2(\Om;\R^N).\end{multline}

We now follow a classical computation; see e.g. \cite[section 4]{francfort.murat}.
Take $\Phi$ in $C^\infty_c(\Om;\R^N)$, $\zeta, \eta\le 1$ in $C^\infty_c(\Om)$ and compute
\begin{multline*}\int_\Om \eta^2\ed \left(\nabla\phd-\sum_{j=1,..,N}\nabla\wdj \Phi_j-\nabla\thd\ \zeta\right)\cdot
 \left(\nabla\phd-\sum_{j=1,..,N}\nabla\wdj \Phi_j-\nabla\thd\ \zeta\right)\ dx=\\[2mm]
 \int_\Om  \eta^2\!\left\{\!\left(\ed\nabla\phd-\td f\right)-\!\!\!\!\sum_{j=1,..,N}\!\!\ed\nabla\wdj \Phi_j-\zeta\left(\ed\nabla\thd\!-\!\td\right)\!\right\}\!\cdot\! \left(\!\nabla\phd-\!\!\!\!\sum_{j=1,..,N}\!\!\!\!\nabla\wdj \Phi_j-\nabla\thd\ \zeta\!\right)
 dx\\[2mm]
 +\int_\Om \eta^2(f-\zeta)\td\cdot\left(\nabla\phd-\!\!\sum_{j=1,..,N}\nabla\wdj \Phi_j-\nabla\thd\ \zeta\right)\ dx\\[2mm]
 =\int_\Om \eta^2\left\{\left(\qd-\!\!\sum_{j=1,..,N}\!\!\ed\nabla\wdj \Phi_j-\zeta\sd\right)+(f-\zeta)\td\right\} \cdot\left(\nabla\phd-\!\!\sum_{j=1,..,N}\!\!\nabla\wdj \Phi_j-\nabla\thd\ \zeta\right)
 \ dx. \end{multline*}
 Multiple applications of the div-curl Lemma, together with \eqref{eq.q}, \eqref{eq.conv.sig}, imply that
 \begin{multline}\label{eq.conv1}
 \int_\Om\eta^2 \left(\qd-\sum_{j=1,..,N}\ed\nabla\wdj \Phi_j-\zeta\sd\right)\cdot\left(\nabla\phd-\sum_{j=1,..,N}\nabla\wdj \Phi_j-\nabla\thd\ \zeta\right)
 \ dx\stackrel{\delta}{\longrightarrow}\\[2mm]\int_\Om\eta^2(q-\e^h\Phi+\zeta a)\cdot(\nabla \ph-\Phi)\ dx=\int_\Om\eta^2(\e^h(\nabla\ph-\Phi)+a(\zeta-f) )\cdot(\nabla \ph-\Phi)\ dx.\end{multline}

 Further,
 \[\td \cdot\nabla\phd=(\td-\ed\nabla \thd)\cdot\nabla \phd+(\ed\nabla \phd-f\td)\cdot\nabla\thd+f\td\cdot\nabla\thd=
 -\sd\cdot\nabla \phd+\qd\cdot\nabla\thd+f\td\cdot\nabla\thd,\]
 so that, setting
 \be{eq.kappa}\kappa:=\fint_\T \tau(y)\cdot\nabla\theta(y)\ dy,\ee
the div-curl Lemma implies, in view of \eqref{eq.conv.sig}, that
 \be{conv.tp}\td \cdot\nabla\phd\rightharpoonup  (a\cdot \nabla\ph+\kappa f), \mbox { {weakly*} in } \Mb(\Om).\ee
 \begin{remark}\label{rem.kappa}Note for later use that, upon multiplication of the first equation in \eqref{eq.eq-t} by $\theta$
 and integration over $\T$, we get
$\kappa= \fint_\T \e(y)\nabla\theta(y)\cdot\nabla\theta(y)\ dy>0.$
 \hfill\P
 \end{remark}

 Hence, since $f$ is in  particular continuous,
  \begin{multline}\label{eq.conv2}
  \int_\Om\eta^2(f-\zeta)\td\cdot\left(\nabla\phd-\sum_{j=1,..,N}\nabla\wdj \Phi_j-\nabla\thd\ \zeta\right)
 \ dx\stackrel{\delta}{\longrightarrow}\\[2mm]  \int_\Om\eta^2(f-\zeta)(a\cdot \nabla\ph+\kappa f-a\cdot\nabla \Phi-\kappa\zeta)\ dx=\int_\Om\eta^2(f-\zeta)
 (a\cdot (\nabla\ph-\Phi)+\kappa (f-\zeta))\ dx.
  \end{multline}
 Summing the contributions \eqref{eq.conv1} and \eqref{eq.conv2}, we finally obtain
\begin{multline*}\lim_\delta\int_\Om \eta^2\ed \left(\nabla\phd-\sum_{j=1,..,N}\nabla\wdj \Phi_j-\nabla\thd\ \zeta\right)\cdot
 \left(\nabla\phd-\sum_{j=1,..,N}\nabla\wdj \Phi_j-\nabla\thd\ \zeta\right)\ dx\\[2mm]=
  \int_\Om\eta^2\left\{(\e^h(\nabla\ph-\Phi)+a(\zeta-f) )\cdot(\nabla \ph-\Phi)+(f-\zeta)
 (a\cdot (\nabla\ph-\Phi)+\kappa (f-\zeta))\right\}\ dx\\[2mm]=
  \int_\Om\eta^2(\e^h(\nabla\ph-\Phi)\cdot(\nabla \ph-\Phi)+\kappa(f-\zeta)
 (f-\zeta))\ dx.
 \end{multline*}
 From this and the coercivity of $\e(y)$ we conclude that, for some $C>0$,
 \begin{multline}\label{eq.limsup}\gamma\limsup_\delta \int_\Om\eta^2 \left(\nabla\phd-\sum_{j=1,..,N}\nabla\wdj \Phi_j-\nabla\thd\ \zeta\right)\cdot
 \left(\nabla\phd-\sum_{j=1,..,N}\nabla\wdj \Phi_j-\nabla\thd\ \zeta\right)\ dx\\[2mm]\le
 C\{\|\nabla\ph-\Phi\|_{L^2(\Om;\R^N)}^2+\|f-\zeta\|_{L^2(\Om{\color{red})}}^2\}.\end{multline}

\medskip

Now, assuming that
\be{hyp.reg}\ba{l}\partial \Om\in C^{2,\alpha},\;0<\alpha<1\\[1mm]\phi\in C^{2,\alpha}(\partial\Om)\\[1mm]f\in C^{1,\alpha}(\Om)\ea\ee
Schauder elliptic regularity applied to \eqref{eq.hom} yields that $\ph\in C^{2,\alpha}(\Omb)$. Then, for any $\lambda>0$ we can find $\phi,\zeta$ such that
\[\|\nabla\ph-\Phi\|_{C^0(\Omb)}+\|f-\zeta\|_{C^0(\Omb)}\le \lambda\]
so that, because $\nabla\wdj,\nabla \thd$ are bounded in $L^2(\Om)$ independently of $\delta$,
\eqref{eq.limsup} together with the arbitrariness of $\eta$ implies that
\[\nabla\phd-\sum_{j=1,..,N}\nabla\wdj \de{\ph}{x_j}-\nabla\thd f\longrightarrow 0,\quad \mbox{ stronly in }L^2_{\rm loc}(\Om{\color{red};}\R^N).\]

We have proved the following
\begin{theorem}\label{thm.clas} Under assumptions \eqref{hyp.reg}, $\phd$, unique
$H^1(\Om)$-solution to \eqref{eq.diel} is such that
\[\nabla\phd-\sum_{j=1,..,N}\nabla\wdj \de{\ph}{x_j}-\nabla\thd f\longrightarrow 0,\quad \mbox{ strongly in }L^2_{\rm loc}(\Om{\color{red};}\R^N)\]
with $\wdj$ defined in \eqref{def.wdj} and $\thd$ defined in \eqref{def.thd}.
\end{theorem}

{\begin{remark}\label{rem.Hconv}Provided that we view the dielectric problem as in \eqref{eq.diel2} -- that is without reference to charges $g^\delta$, but with an oscillating field $\tau^\delta$ -- the obtained results do not rest on periodicity and could be adapted to a general homogenization framework like that of $H$-convergence \cite{murat.tartar}.\hfill\P\end{remark}}

\begin{remark}\label{rem.global} If we also assume that
\be{hyp.g}
g\in C^{0,\alpha}(\T)
\ee
 for some $0<\alpha<1$, then elliptic regularity applied to \eqref{eq.psi} implies that $\tau \in C^{1,\alpha}(\T;\R^N)$, and thus that, in particular, the convergence in \eqref{conv.tp} takes place weakly in $L^2(\Om).$ Because of that, we do not need the compactly supported smooth test $\eta$ in all prior computations and the result of Theorem \ref{thm.clas} becomes
\[\nabla\phd-\sum_{j=1,..,N}\nabla\wdj \de{\ph}{x_j}-\nabla\thd f\longrightarrow 0,\quad \mbox{ strongly in }L^2(\Om{\color{red};}\R^N).\]
Note that the fact that $\phd$ satisfies Dirichlet boundary conditions is essential in evaluating the limit of the term $\int_\Om \ed\nabla\phd\cdot\nabla\phd\ dx$ in all previous computations.
\hfill\P\end{remark}

\begin{remark}\label{rem.mat-at-bdary} In the case of a two-phase particulate microstructure, that is whenever
\[\e(y):=\chi_\M\e_\M+(1-\chi_\M)\e_\I, \mbox{ with } \I:=\T\setminus\M\]
where $\M$ is a measurable subset of $\T$, we can modify the definition of $\ed$ so that no inclusion intersects $\partial \Omega$.  We define
\[\Omega_\delta := \bigcup_{z \in \mathbb Z^N \mbox{\tiny\ s.t. } \delta (z+2Y)\subset \Omega} \delta(z+2Y),\]
note that
$|\Omega \setminus \Omega_\delta|\le C \delta$ and set
\[\ed(x):=\e(\frac x\delta)\chi_{\Omega_\delta}+\e_\M\chi_{\Omega\setminus \Omega_\delta}.\]
With this definition of $\ed$, the results of Section~\ref{sec.ch} still hold true with trivial modifications of the proofs. \hfill\P
\end{remark}

Unfortunately, as described in the introduction, this result is not sufficient, when plugged into the equations of elasticity, to ensure that $\nabla \phd$ can be replaced by $\sum_{j=1,..,N}\nabla\wdj \de{\ph}{x_j}+\nabla\thd f$ in those, or that one can perform any kind of homogenization process on the resulting system.
This is why the next Section is devoted to an improvement of Theorem \ref{thm.clas}. The framework required for the successful completion of such a task will be much more constrained than that in the current Section. In particular, periodicity, which was merely convenience so far, will become essential. Equally essential will be the assumption that the microstructure is {particulate}.

\section{Improved estimates and correctors result for the dielectrics}\label{sec.ic}
Throughout this Section, we assume that $\partial \Om\in C^{2,\alpha}$ for some $0<\alpha<1$.
The unit torus $\T$ is of the form
\be{hyp.phase}\ba{l}\T= \M\cup \I,\; \M \mbox{ closed such that } \partial \M \mbox{ is }C^{1,\beta},  \;0<\beta<1\\[2mm] \I=\T\setminus \M\\[2mm]
\overline{\mathbf i(\M)}  \mbox{ is a connected subset of $\R^N$ (a matrix phase).}\ea\ee
Further,
\be{hyp.coef}\e(y):=\chi_\M\e_\M+(1-\chi_\M)\e_\I, \mbox{ with } \gamma |\xi|^2\le \e_{\M,\I}\xi\cdot\xi\le\gamma'|\xi|^2.\ee
We  define $\ed$ as in Remark \ref{rem.mat-at-bdary}.

  For some $0<\alpha < \frac{\beta}{(\beta+1)N}$ we also assume that $f\in C^{1,\alpha} (\Om)$,  that $\psi$, the solution to \eqref{eq.psi}, satisfies
  \be{hyp.psi}\psi \in C^{1,\alpha}(\T;\R^N),\ee
and that $\phi$ is the restriction to $\partial\Om$ of a function of $C^{1,\alpha}(\Omb)$ (still denoted by $\phi$).
Remark that the assumed regularity of $\psi$ will be achieved if , e.g., $g\in C^{0,\alpha}(\T)$.

In a first step, we prove $\delta$-independent $L^q$-estimates on $\nabla \phd$ for any
$1\le q<\infty$. This is the object of the following
\begin{proposition}\label{prop.Lqreg}
For all $1\le q <\infty$, the sequence $\nabla\phd$ is bounded in $L^q(\Om;\R^N)$, independently of $\delta$.
%
\end{proposition}
\begin{proof}

\noindent{\sf Step 1.} First, we apply the large-scale Calder\'on-Zygmund estimates of \cite[Theorem~7.7]{AKM-book}   to \eqref{eq.diel2} rewritten as
\begin{equation}\label{eps2}
\left\{
\begin{array}{rcl}
\dvg \ed \nabla (\phd-\phi)&=&\dvg (f \td-\ed \nabla \phi+\nabla \zd) \\[2mm]
\phd-\phi&=&0 \quad  \mbox{ on }\partial \Om,
\end{array}
\right.
\end{equation}
where $\zd$ is the unique $H^1_0(\Om)$-solution to
\[\triangle \zd=\nabla f\cdot\td.\]
We get, for all $\delta>0$ and $q\ge 2$,
\begin{multline*}
\int_{\Om} \Big(\fint_{B_\delta(x)} \chi_\Om(z) |\nabla (\ph^\delta-\phi)|^2(z)\ dz\Big)^\frac q2\ dx \\[2mm] \le \, C_q  \Big( \| \nabla (\ph^\delta-\phi) \|_{L^2(\Om)}^q+ \|f\td-\ed\nabla \phi+\nabla \zd\|_{L^q(\Omega)}^q\Big).
\end{multline*}

Note that, since the coefficients are periodic, the random variable involved in  \cite[Theorem~7.7]{AKM-book}
is simply a constant (and in fact $\delta$) in the present setting.

Now, since $\psi$ is in $C^{1,\alpha}(\T;\R^N)$,
{$\td$ given by \eqref{eq.tau} is bounded in
$L^q(\Om;\R^N)$ independently of $\delta$ and
maximal $L^q$-regularity for the Laplacian
implies that $\nabla \zd$ is bounded in $W^{1,q}(\Om;\R^N)$
independently of $\delta$.
}

Since, by Jensen's inequality,
\[
\int_{\Omega} \Big(\fint_{B_\delta(x)} \chi_\Omega |\nabla \phi|^2\ dz\Big)^\frac q2\,dx\le\,
 \int_{\R^N} \fint_{B_\delta(x)} \chi_\Omega |\nabla \phi|^q
\,dzdx= \, \int_{\R^N} \chi_\Omega |\nabla \phi|^qdx = \int_{\Omega}   |\nabla \phi|^q\ dx,
\]
 the first convergence in \eqref{eq.wconv} and the assumed regularity of the functions $f$ and $\psi$ finally yield
\be{eq.estq2}
\int_\Om\left(\fint_{B_\delta(x)} \chi_\Omega(z) |\nabla \ph^\delta|^2\ dz\right)^\frac q2\,dx\le C_q
\ee
for some constant $C_q$ depending on $q$ and on $\|f\|_{C^{1,\alpha}(\Om)},\|\psi\|_{C^{1,\alpha}(\T)},\|\phi\|_{C^{1,\alpha}(\Om)}.$ Estimate \eqref{eq.estq2} enables us to control $\nabla \phd$ on scales larger than $\delta$. In a second step, we will derive an estimate for small scales, that is for scales smaller than $\delta$.

\medskip
\noindent{\sf Step 2.} In this step, we crucially use the two-phase character of the
microstructure, as well as the $C^{1,\beta}$-regularity of the boundary of each of those phases. Take a point $x\in \Om$ and consider { the cube $Q_{2\delta}(x)$ of side-length $2\delta$ centered at $x$}. We blow up equation \eqref{eq.diel2} so as to obtain an equation on {$Q_2(0)$}.

To that effect, we set
\[\Phi^\delta(z):=\frac1\delta\phd(x+\delta z).\]
Then $\Phi^\delta$ satisfies
\be{eq.diel-res}\dvg_z \left(\e(\frac x\delta+z) \nabla_z\Phi^\delta-\!f(x+\delta z)\td(x+\delta z)\right)\!= -\delta\td(x+\delta z)\!\cdot\!\nabla f(x+\delta z) \mbox{ in } (\Om-\{x\})/\delta\cap {Q_2(0)}.\ee
{Since $\partial \Omega$ is $C^{2,\alpha}$ and $\partial \M$ is $C^{1,\beta}$, an elementary geometric argument shows there is {\color{red} a} domain $\Omega'$ containing $Q_1(0)$ such that $\partial \Big( (\Om-\{x\})/\delta\cap \Omega'\Big)$ is $C^{1,\beta}$
uniformly in $x$ and $\delta$.
Now, the set $(\Om-\{x\})/\delta\cap \Omega'$  contains a number of inclusions bounded uniformly wrt $\delta$ : if there are $p$ inclusions in $\T$, that is if $\I$ has $p$ connected components, then that set contains at most $2^Np$ inclusions. Since by assumption the boundary of those inclusions is also $C^{1,\beta}$, then, in the terminology of \cite{LV-00}, the $C^{1,\beta}$ modulus $K$ of $(\Om-\{x\})/\delta\cap \Omega'$ does not depend on $\delta$ and \cite[Theorem 1.1]{LV-00} applies.
Considering, for $\eta>0$, the set $\Om'_{\delta,\eta}=\{z\in \Om-\{x\})/\delta\cap Q_1(0):
{\rm dist}(z,\partial(\Om-\{x\})/\delta\cap Q_1(0))>\eta\}$}
we conclude in particular that, for any $c\in \R$,
\be{eq.int-ss}\|\nabla_z \Phi^\delta \|_{L^\infty(\Om'_{\delta,\eta})}\le
C_{K,\eta}\left(\|\Phi^\delta-c\|_{L^\infty(\Om'_{\delta,\eta/2})}+ C'\right),\ee
where $C_{K,\eta}$ is a constant that only depends on $K, \alpha,\beta, \gamma,\gamma',\eta$ while $C'$ is a constant that only depends on $\|f\|_{C^{1,\alpha}(\Om)},\|\psi\|_{C^{1,\alpha}(\T)},\|\phi\|_{C^{1,\alpha}(\Om)}.$

We now apply De Giorgi-Nash-Moser's theorem \cite[Theorem 8.24]{gilbarg.trudinger} to \eqref{eq.diel-res}. This yields in turn the following interior estimate:
\be{eq.DGNM} \|\Phi^\delta-c\|_{L^\infty(\Om'_{\delta,\eta/2})}\le C_\eta(\|\Phi^\delta-c\|_{L^2((\Om-\{x\})/\delta\cap Q_1(0))}+ C''),\ee
where $C_\eta$ depends only on $\eta, \gamma,\gamma'$  and $C''$ depends only on $\|f\|_{C^{1,\alpha}(\Om)}$, $\|\psi\|_{C^{1,\alpha}(\T)},\|\phi\|_{C^{1,\alpha}(\Om)}.$
Inserting \eqref{eq.DGNM} into \eqref{eq.int-ss} yields
\[\|\nabla_z \Phi^\delta \|_{L^\infty(\Om'_{\delta,\eta})}\le
C'_{K,\eta}\left(\|\Phi^\delta-c\|_{L^2((\Om-\{x\})/\delta\cap {Q_1(0)})}+ C'''\right),\]
where $C'_{K,\eta}$ is a constant that only depends on $K, \alpha,\beta, \gamma,\gamma',\eta$ while $C'''$ is a constant that only depends on $\|f\|_{C^{1,\alpha}(\Om)},\|\psi\|_{C^{1,\alpha}(\T)},\|\phi\|_{C^{1,\alpha}(\Om)}.$

Now choose $c:=\fint_{L^2((\Om-\{x\})/\delta\cap {Q_1(0)})}\Phi^\delta(z)\ dz$ and apply Poincar\'e-Wirtinger's inequality to the previous estimate. We obtain
\be{eq.est-ss}\|\nabla_z \Phi^\delta \|_{L^\infty(\Om'_{\delta,\eta})}\le
C'_{K,\eta}\left(\|\nabla\Phi^\delta\|_{L^2((\Om-\{x\})/\delta\cap {Q_1}(0))}+ C'''\right).\ee

Blowing \eqref{eq.est-ss} down, we conclude in particular that, {for a $d>0$ large enough, we have, for
all $\delta>0$, }
\be{eq.est-local}
 \sup_{B_{\delta/{d}}(x)\cap \Om}|\nabla \phd|\le C'_{K,\eta}\left(\left(\fint_{\Om\cap B_\delta(x)}|\nabla\phd(y)|^2\ dy\right)^\frac12+ C'''\right).\ee

\medskip

\noindent{\sf Step 3.} We combine the estimates obtained in the first two steps as follows.

Remark that, for all $x\in \Om$ and all small enough $\delta$'s, the assumed regularity of $\partial\Om$ implies the existence of $d>0$ such that $|B_\delta(x)\cap\Om|\ge d|B_\delta(x)|$.  From  \eqref{eq.estq2} we then get
\[\int_\Om\left(\fint_{\Om\cap B_{\delta}(x)}|\nabla\phd(y)|^2\ dy\right)^\frac q2 dx\le d^{-\frac q 2}\int_\Om\left(\fint_{B_{\delta}(x)}\chi_\Om(y)|\nabla\phd(y)|^2\ dy\right)^\frac q2 dx\le d^{-\frac q 2}C_q.\]

In turn, from \eqref{eq.est-local}, we obtain
\begin{multline*}\int_\Om|\nabla\phd|^q\ dx\le \int_\Om\left(\sup_{B_{\delta/{d}}(x)\cap \Om}|\nabla\phd|^q\right) dx\\[2mm]
\le 2^q(C'_{K,\eta})^q \left(\int_\Om\left(\fint_{\Om\cap B_\delta(x)}|\nabla\phd(y)|^2\ dy\right)^\frac q2 dx+ |\Om|(C''')^q\right).\end{multline*}

Combining the  inequalities above, we finally conclude that
\[\int_\Om|\nabla\phd|^q\ dx\le 2^q(C'_{K,\eta})^q \left(d^{-\frac q 2}C_q+ |\Om| (C''')^q\right).\]
This completes the proof of the proposition.
\end{proof}

We are now in a position to improve on the convergence result of Theorem \ref{thm.clas}. We obtain the following {\color{red} .}

\begin{theorem}\label{thm.impr}
Under assumptions \eqref{hyp.reg}, \eqref{hyp.g}, \eqref{hyp.phase}, \eqref{hyp.coef}, $\phd$, unique
$H^1(\Om)$-solution to \eqref{eq.diel} is such that, for any $1\le q<\infty$,
\[\nabla\phd-\sum_{j=1,..,N}\nabla\wdj \de{\ph}{x_j}-\nabla\thd f\longrightarrow 0,\quad \mbox{ stronly in }L^q(\Om,\R^N)\]
with $\wdj$ defined in \eqref{def.wdj} and $\thd$ defined in \eqref{def.thd}.
\end{theorem}

\begin{proof}
First note that the regularity assumptions on the domain, $f$ and $\phi$ and classical Schauder regularity imply that $\ph$, the solution to \eqref{eq.hom}, is in $C^{2,\alpha}(\Om)$. Further, the regularity assumption on $g$ and  classical Schauder regularity imply that $\tau\in C^{1,\alpha}(\T)$. Then another application of \cite[Theorem 1.1]{LV-00}, this time on $\T$ which has no boundary, implies in particular  that
\be{eq.linfireg}\nabla\theta,  \nabla \chi_j\in L^\infty(\T),\ee
 hence $\nabla w_j$
as well.  We can thus assume that, for any $r\ge 1$, the term
\[\sum_{j=1,..,N}\nabla\wdj \de{\ph}{x_j}+\nabla\thd f \mbox{ is bounded in }L^r(\Om) \mbox{ independently of }\delta.\]

Set $\theta=1/(q-1)$. In view of Remark \ref{rem.global} and Proposition \ref{prop.Lqreg}, the uniform bound derived above yields that, for some constant $C$ depending on $q$ and all the data,
\begin{multline*}\|(\nabla\phd-\sum_{j=1,..,N}\nabla\wdj \de{\ph}{x_j}-\nabla\thd f)\|_{L^q(\Om)}\le \|(\nabla\phd-\sum_{j=1,..,N}\nabla\wdj \de{\ph}{x_j}-\nabla\thd f)\|^\theta_{L^2(\Om)}\times\\[2mm]  \|(\nabla\phd-\sum_{j=1,..,N}\nabla\wdj \de{\ph}{x_j}-\nabla\thd f)\|^{1-\theta}_{L^{2q}(\Om)}\le C \|(\nabla\phd-\sum_{j=1,..,N}\nabla\wdj \de{\ph}{x_j}-\nabla\thd f)\|^\theta_{L^2(\Om)}\stackrel{\delta}{\longrightarrow}0.
\end{multline*}
Hence the result.
\end{proof}

\begin{remark}\label{rem.mult-ph} We have assumed throughout this Section that the composite is made of two phases. Nothing would change if we considered $n$ phases instead of 2, provided that we keep the same regularity assumptions. Theorem \ref{thm.impr} would still hold true, and Theorem
\ref{thm.hom-elast} below as well.
\hfill\P\end{remark}

\begin{remark}\label{rem.source-term} All results of Sections  \ref{sec.ch}, \ref{sec.ic} remain valid if a source term of the form $h^\delta v$ is added to the right hand-side of \eqref{eq.system} with $h\in L^\infty(\T)$ and
$v\in L^\infty(\Om)$. Then, one has to add the term $\left(\int_\T h(y)\ dy\right)  v$  to the right hand-side of the homogenized equation \eqref{eq.hom}. All other results remain unchanged.\hfill\P\end{remark}

\section{Homogenization of the elasto-dielectrics}\label{sec.hed}

We now address the elasticity part of the problem. Recall that $\Ld(x):=L\xd$ where $L(y)$ is a measurable, symmetric linear mapping from $\Msym$ into itself with the properties that $\gamma |e|^2\le L(y)e\cdot e\le \gamma' |e|^2$ for a.e. $y\in\T$ and some
$0<\gamma<\gamma'<\infty$. Also $\Md(x):=M\xd$ with $M(y)$ a bounded, measurable,  linear mapping from $\Msym$ into itself.

 The equations are

\be{eq.elast}\ba{l} \dvg (\Ld \nabla \ud +\Md (\nabla\phd\otimes\nabla\phd))=0\\[2mm]
\ud=0 \mbox{ on }\partial\Om.\ea\ee

We assume that assumptions \eqref{hyp.reg}, \eqref{hyp.g}, \eqref{hyp.phase}, \eqref{hyp.coef} (hence also \eqref{hyp.psi}) hold true throughout this Section.

In particular, we can apply Proposition \ref{prop.Lqreg} and we conclude, with the help of Korn and Poincar\'e inequalities, that
\[\ud \mbox{ exists and is bounded in } H^1_0(\Om;\R^N) \mbox{ independently of }\delta.\]

We can also apply Theorem \ref{thm.impr} and we immediately obtain that, for any $0<r<\infty$,
\be{eq.conv.rhs} \Md (\nabla\phd\otimes\nabla\phd) - \Md [(\!\!\sum_{j=1,..,N}\!\nabla\wdj \de{\ph}{x_j}+\nabla\thd f)\otimes (\!\!\sum_{j=1,..,N}\!\nabla\wdj \de{\ph}{x_j}+\nabla\thd f)]
\stackrel{\delta}{\longrightarrow}0,\mbox{ strongly in } L^r(\Om).\ee

We will only use the value $r=2$ hereafter. Set
\be{eq.xd}Z^\delta(x):=\Md [(\sum_{j=1,..,N}\nabla\wdj \de{\ph}{x_j}+\nabla\thd f)\otimes (\sum_{j=1,..,N}\nabla\wdj \de{\ph}{x_j}+\nabla\thd f)].\ee
Because of convergence \eqref{eq.conv.rhs},  if $\tud$ is the unique $H^1_0$- solution to
\be{eq.elast-tilde}\ba{l} \dvg (\Ld \nabla \tud +Z^\delta)=0\\[2mm]
\ud=0 \mbox{ on }\partial\Om,\ea\ee
then
\be{eq.conv-dif}\ud-\tud \stackrel{\delta}{\longrightarrow}0,\;\mbox{ strongly in } H^1_0(\Om;\R^N).\ee

\medskip

We undertake a homogenization process for the system \eqref{eq.elast-tilde}.
To that effect, we introduce
 the periodic corrector $W_{ij}$ defined as follows. Set $X_{ij}$ to be the unique solution in $H^1(\T;\R^N)$ to
\[\ba{l}\dvg L \nabla (X_{ij}+x_i\vec{e}_j)=0\\[2mm]\int_\T X_{ij}\ dy=0.\ea\]
Then set
\[W_{ij}(y):=X_{ij}+y_i\vec{e}_j.\]
and
\[\Wdj(x):= \delta X_{ij}\xd+x_i\vec{e}_j\]
and note that $\nabla\Wdj(x)=(\nabla W)\xd$.

Then an argument near identical to that which led to \eqref{eq.hom} would yield that
\be{conv.tud}\tud \weak u\; \mbox{ in } H^1_0(\Om;\R^N),\ee
with $u$, unique $H^1_0(\Om;\R^N)$-solution to
\be{eq.hom-el}
\ba{l} \dvg (L^h \nabla u +Z)=0\\[2mm]
u=0 \mbox{ on }\partial\Om,\ea\ee
with $L^h$ defined as
\be{elast.hom}L^h_{ijkh}:= \fint_\T L(y)\nabla W_{ij}\cdot\nabla W_{kh}\ dy\ee
and $Z$, the $L^2(\Om;\Msym)$-weak limit of $Z^\delta$ defined in \eqref{eq.xd}, being
\be{def.Z}
Z:=M^h(\nabla \ph\otimes \nabla \ph)+2fN^h \nabla \ph  +P^h f^2
\ee
where
\be{def.MNP}\ba{lcl}M^h_{ijkh}&:=&\ds \fint_\T M(y)(\nabla{w_k}(y)\otimes\nabla{w_h}(y))\cdot\nabla W_{ij}(y)\ dy\\[4mm] N^h_{ijk}&:=&\ds \fint_\T M(y)(\nabla{w_k}(y)\otimes\nabla\theta(y))\cdot\nabla W_{ij}(y)\ dy\\[4mm] P^h_{ij}&:=&\ds \fint_\T M(y)(\nabla{\theta}(y)\otimes\nabla{\theta}(y))\cdot\nabla W_{ij}(y)\ dy.
\ea\ee
Convergences \eqref{eq.conv-dif} and \eqref{conv.tud} imply the following homogenization result
\begin{theorem}\label{thm.hom-elast}
 Under assumptions \eqref{hyp.reg}, \eqref{hyp.g}, \eqref{hyp.phase}, \eqref{hyp.coef}, $\ud,$ unique
$H^1(\Om;\R^N)$-solution to \eqref{eq.elast}, converges weakly in $H^1(\Om;\R^N)$ to the unique $H^1_0(\Om;\R^N)$-solution $u$ to \eqref{eq.hom-el} with $L^h$ defined in \eqref{elast.hom} and $Z$ defined through \eqref{def.Z}, \eqref{def.MNP}.
\end{theorem}

\begin{remark}  Note that, since $X_{ij}=X_{ji}$, $M^h$ enjoys the same symmetry properties as $M(y)$, that is $M^h_{ijkh}=M^h_{jikh}=M^h_{ijhk}$ while $N^h_{ijk}=N^h_{jik}$ and $P^h_{ij}=P^h_{ji}$. Of course $L^h$ has the usual symmetries of elasticity, that is $L^h_{ijkh}=L^h_{jikh}=L^h_{khij}$.
\hfill\P\end{remark}

\begin{remark} We could, in the spirit of the previous Sections, provide a corrector result for $\ud$ but will refrain from doing so because of the notational complexity.
\hfill\P\end{remark}

\begin{remark} Our result  should be compared to that in \cite{TTBB},
 which investigates the case $g=0$. In such a case $N^h=P^h=0$ and the corrector results of the previous sections are markedly simpler. Also note  that the results in \cite{TTBB} are derived under the {\em a priori} assumption that $\nabla \phd$ is bounded in $L^4(\Om;\R^N)$ independently of $\delta$; no justification for such an estimate is offered in that work.
\hfill\P\end{remark}

\section{The case for active charges}\label{sec.ac}

\subsection{Setting of the problem}

The homogenized dielectric equation obtained in  \eqref{eq.hom} can be equivalently rewritten as
\[ \dvg (\e^h\nabla \ph -a f)=0.\]
Now, in practice there should be several collections of charges, that is that, in lieu of a charge of the form $g^\delta f$, one should envision a charge of the form
$\sum_{p=1,...,N} g_p^\delta f_p$ where each pair $(g_p,f_p)$ is endowed with the same properties, namely $\int_\T g_p(y)\ dy=0$, and the necessary regularities of $g_p$ and $f$ that were introduced in the previous Subsections.

Provided that those are met, the homogenization results remain unchanged by linearity. In particular the homogenized dielectric equation becomes
\be{eq.hmult}\ba{l}
\dvg (\e^h\nabla \ph -\sum_{p=1,...,N}a_p f_p)=0\\[2mm]\ph=\phi \mbox{ on }\partial\Om,\ea\ee
with $a_p\in\R^N$ defined as (see~\eqref{eq.defaj})
\be{eq.apj}(a_p)_j:=a_{jp}:=\fint_\T \tau_p(y)\cdot\nabla w_j(y)\ dy=\fint_\T \nabla\psi_p(y)\cdot(\vec{e}_j+\nabla \chi_j(y))\ dy=-\fint_\T g_p(y)\chi_j(y)\ dy.
\ee
In \eqref{eq.apj}, $\tau_p=\nabla \psi_p$ with $\psi_p$ defined as $\psi$ was in \eqref{eq.psi} upon replacing $g$ by $g_p$.

We can thus view $a$ as a $N\times N$-matrix with $j,p$ coefficient $a_{jp}$.
Note that that matrix is not necessarily symmetric.

\medskip

{\em Active charges}, if they exist, consist in an appropriate choice of $f_p$ so that the homogenized dielectric displays an enhancement ({or degradation}) of its permittivity. In other words, one would like to choose $f_p={\partial\ph}/{\partial x_p}$ so that $\ph$ is the solution to \eqref{eq.hmult}. Then \eqref{eq.hmult} reads as
\be{eq.henh}\ba{l}\dvg \teh\nabla \ph=0\\[2mm]\ph=\phi \mbox{ on }\partial\Om\\[2mm]\teh:=\e^h-a.\ea\ee
Furthermore, if desiring electric enhancement, and not electric degradation, one should ensure that $\teh$ admits at least one positive eigenvalue with a value greater than those of $\e^h$.

\medskip

In this two-step process, one should first ensure  existence of active charges, that is existence of a field $\ph$ that satisfies \eqref{eq.henh}. This amounts to choosing $g$ so that $\teh$ is strongly elliptic.
To this aim, take
$$
g_p\equiv \omega \chi_p-\int_\T\omega(y) \chi_p(y)dy
$$
where $\omega\ge 0 \in C^\infty(\T)$ with support $\bar \S$, $\S\subset\T$ open, and note  that, provided that \eqref{hyp.reg},  \eqref{hyp.phase}, \eqref{hyp.coef} hold true,
we are indeed in the setting of Section \ref{sec.ic}.  Then from \eqref{eq.apj}
\[
a_{jp}=a_{pj}=-\fint_\T \omega\chi_j(y)\chi_p(y) \ dy,
\]
so that, for any $\xi\in\R^N\ne 0$, $-\sum_{j,p=1,...,N}a_{jp}\xi_j\xi_p=\ds\fint_\T \omega(\sum_{k=1,...,N}\chi_k(y)\xi_k)^2\ dy>0$ unless $\chi(y)\perp \xi$, a.e. on $\S$, where
$\chi(y)$ is the vector with components $\chi_k(y), \;k=1,...,N$. In that case, taking $\xi/|\xi|$ as first unit vector we find that $\chi_1(y)\equiv 0$ on $\S$ and thus, recalling \eqref{eq.chij} and the canonical identification between $\T$ and the  $Y$-periodic paving of $\R^N$, that,
$$
(\e_\M-\e_I)\nu_1(y)=0 \mbox{ on } \partial M\cap\partial I\cap S
$$
where $\nu$ is the exterior normal to $\partial M$. So, in $S$, $M$ is a cylinder with axis parallel to $\xi$.  Consequently,  unless $\nu(y)\perp\xi$ for all $y\in\partial M$, it will always be so that
\be{eq.pos-a} a \mbox{ is symmetric negative definite}.\ee
But the normal cannot be always perpendicular to $\xi$ because, in such a case, $\M$ would not satisfy the last assumption in \eqref{hyp.phase}.

Upon multiplication of $g_p=\omega\chi_p$ by a large enough factor $\lambda$ independent of $p$ we conclude to the existence of large enough charges such that
$\teh$ is a symmetric positive definite matrix whose eigenvalues can be arbitrarily large upon choosing $\lambda$ large enough. Thus, we can always solve
\eqref{eq.henh} and obtain a large enhancement.

We have proved the following

\begin{proposition}\label{prop.enh} Under assumptions   \eqref{hyp.reg}, \eqref{hyp.phase}, \eqref{hyp.coef}, \eqref{eq.pos-a} always holds true if
choosing  $g_p$  to be $\lambda \chi_p$ for $p=1,..., N$ and $\lambda>0$. Then, $\teh$ can have arbitrarily large {positive} eigenvalues with an appropriate choice of $\lambda$.
\end{proposition}
Of course, Proposition \ref{prop.enh} provides no  answer to the more useful question of finding a manufacturable set of micro-charges such that enhancement can occur.
Indeed, Proposition \ref{prop.enh} relies on the knowledge of the correctors $\chi_p$.
There are two possible routes if one wishes to take advantage of this observation in practice.
On the one hand, one can resort to numerical simulations of $\chi_p$,  and use the output to devise a distribution of charges that will enhance the permittivity.
On the other hand, one can resort to asymptotic analysis for some specific geometries. This is what we will do in the next subsection devoted to the case
of periodically distributed small charged inclusions.

\subsection{Enhancement  for dilute inclusions}

The following is inspired by the Clausius-Mossotti formula for dilute spherical inclusions.
{In this paragraph, $C$ denotes a finite constant that may change from line to line but which is independent of $\lambda$ and $\ell$ (see below).}
Let $B$ denote the unit ball, and $B_{1+\eta}$ the ball of radius $1+\eta$ for some fixed $\eta>0$.
We shall confine charges to the intermediate phase $B_{1+\eta} \setminus B$ (the coating).
We need an additional scale $\ell \ge 1$ to quantify dilution, and define  $\e_\ell$
to be the permittivity tensor associated with  the $\ell$-periodic extension of the map defined on $\ell [-\frac12, \frac12)^N$ by
$$
\e_\ell(y) := (1+ (\bar \e-1)\chi_{B}(y))\Id,
$$
which models a background medium of permittivity $1$ perturbed by spherical inclusions of a medium of permittivity $\bar \e$ centered on the grid $(\ell \Z)^d$. In particular, the density of inclusions is $\ell^{-d}|B|$.  We correspondingly set in $\R^N$
$$\e_\infty(x):= (1+ (\bar \e-1)\chi_{B}(x))\Id.$$
When $\ell$ is very large, Clausius and Mossotti {\cite{Clausius-79,Mossotti-36,Mossotti-50}} argued that at first order the spherical inclusions do not interact with each other. Denote by $\chi_{\ell p}$ the $\ell$-periodic corrector associated with $\e_\ell$ and direction $\vec{e_p}$  as in \eqref{eq.chij}, and
denote by $\chi_{\infty p}$ the solution of {the ``Maxwell'' \cite[Chapter IX]{Maxwell54} or ``Eshelby'' problem \cite{Eshelby57}}
\be{eq.eshelby}
\ba{l}\dvg \e_\infty \nabla (\chi_{\infty p}+x_p)=0\\[2mm] \chi_{\infty p}(x) \stackrel{|x|\to \infty}{\longrightarrow}0 .
\ea\ee
Equation~\eqref{eq.eshelby} can be solved explicitly, yielding,
for $|x| \ge 1$
\be{fo.CM}
\chi_{\infty p}(x)= \Big(\frac{1-\bar \e}{\bar \e +N-1} \Big) \frac{x_p}{|x|^N}.
\ee
On the other hand, as proved  in \cite{Pertinand}, $\chi_{\ell,p}$ is close to $\chi_{\infty p}$ on $B_{1+\eta}$ in the sense that there exists some constant $C>0$ such that, for  $\ell \gg 1$ we have
\be{Jules}
\| \nabla (\chi_{\infty p}-\chi_{\ell p})\|_{L^2(B_{1+\eta})} \, \le \,C \ell^{-N}.
\ee
We then take $g_{\ell p}$ as the $\ell$-periodic extension of the map defined on $\ell [-\frac12, \frac12)^N$ via
$$
{g_{\ell p}(y):=
\chi_{\infty p}(y) \chi_{B_{1+\eta}\!\setminus\! B}(y),}
$$
which has vanishing average in view of  \eqref{fo.CM}.
Define the matrix $\bar a$ by
$$
\bar a_{pj}:=\int_{B_{1+\eta}\setminus B} \chi_{\infty p}(y)\chi_{\infty j}(y)\ dy,\; 1\le p,j\le N,
$$
so that
\begin{eqnarray}
\bar a&=& \Big(\frac{1-\bar \e}{\bar \e +N-1} \Big)^2 1/N \int_{B_{1+\eta}\setminus B} |x|^{2(1-N)} \ dx \;\Id
\nonumber \\
&=&\Big(\frac{1-\bar \e}{\bar \e +N-1} \Big)^2   \left\{
\begin{array}{rcl}
 \eta  &:&N=1\\
\pi \log(1+\eta)&:&N=2 \\
\frac{|\mathcal S^{N-1}|}{N}(1-(1+\eta)^{2-N}) &:&N>2
\end{array}
\right\}\;\Id
,\label{eq.inf-sou}
\end{eqnarray}
{where $|\mathcal S^{N-1}|$ denotes the surface of the unit sphere in dimension $N$}.
The combination of \eqref{fo.CM} and \eqref{Jules}  together with the Poincar\'e-Wirtinger inequality allows us to conclude that $a_\ell$ defined as in
\eqref{eq.apj} is quantitatively close to $-\ell^{-N}\bar a$, namely, for some constants $C,C'$,
$$
|a_{\ell pj}+\ell^{-N}\bar a_{pj}| = \ell^{-N}  \Big| \int_{B_{1+\eta}\setminus B} \chi_{\infty p}(y)(\chi_{\infty j}-\chi_{\ell j})\ dy\Big| +C\ell^{-2N} \,
\le \, C' \ell^{-2N},
$$
from which we deduce that for all $\ell \gg 1$ large enough, $a_\ell$ is diagonalizable with
negative eigenvalues of order $\ell^{-N}$.
Upon multiplying $g_{\ell p}$ by the factor $\ell^N \lambda$ for some $\lambda \gg 1$,
one obtains for $\teh_\ell$ defined as in \eqref{eq.henh} (and $\e^h_\ell$ defined as in \eqref{eq.hom-diel})
\[\teh_\ell:=\e^h_\ell+\lambda \bar a+O(\ell^{-N}\lambda),\]
with $\bar a$ defined in \eqref{eq.inf-sou}, thus ensuring some enhancement of strength $\lambda$. Note that here $\e^h_\ell$ is the homogenized   permittivity associated with a spherical inclusion of radius 1 in a cell of side-length $\ell$, or, equivalently, a spherical inclusion of radius $1/\ell$ in the unit cell $Y$.

\medskip

Let us give the leading order part of $f$ in the regime $\lambda,\ell^N\gg 1$. Since $\bar a$ is a multiple of the identity, the equation for $\ph$ takes the form
\be{eq.enhanced}\ba{l}\dvg (\rm{\;\sf i}+\kappa_{\lambda\ell})\nabla \ph=0\\[2mm]\ph=\phi \mbox{ on }\partial\Om\ea\ee
with $\|\kappa_{\lambda\ell}\|_\infty \le C(\ell^{-N}+\lambda^{-1})$ so that $f_p ={\partial\tilde \ph}/{\partial x_p}+O(\ell^{-N}+\lambda^{-1})$ (as an identity in $H^1(\Omega)$) with $\tilde \ph$ solving
\be{eq.enhanced-leading}\ba{l}-\triangle \tilde \ph=0\\[2mm]\tilde \ph=\phi \mbox{ on }\partial\Om.\ea\ee

\medskip

We conclude this subsection with an investigation of the impact of the dielectric enhancement on the elastic response of dilute inclusions.

In the statement of Theorem~\ref{thm.hom-elast}, which yields the homogenized equation for the displacement field, the only contribution of the active charges
{$\sum_{p=1}^N \lambda \ell^N g_{\ell p}$} are the terms $\sum_{p=1}^N (2f_p N^{p h}_{\lambda \ell} \nabla \ph+P^{p h}_{\lambda\ell} f_p^2)$ involving the tensors $N^{p h}_{\lambda\ell}$ and $P^{p h}_{\lambda\ell}$ in the forcing $Z$; cf.~\eqref{def.Z} and \eqref{def.MNP} (the sum over $p$ follows by linearity).
In the setting of the example of dilute inclusions, these tensors take the form
\be{def.MNP+}\ba{lcl}
(N^{h p}_{\lambda \ell})_{ijk}&:=&\ds  \lambda \ell^{N}\fint_{\ell\T} M_\ell(y)(\nabla{w_{\ell k}}(y)\otimes \nabla\theta_{\ell p}(y))\cdot\nabla W_{\ell ij}(y)\ dy\\[4mm]
 (P^{h p}_{\lambda \ell})_{ ij}&:=&\ds (\lambda \ell^{N})^2 \fint_{\ell \T} M_\ell(y)(\nabla{\theta^{p}_{\ell}}(y)\otimes\nabla{\theta_{\ell p}}(y))\cdot\nabla W_{\ell ij}(y)\ dy.
 \ea\ee
where $\theta_{\ell p}$ solves (see \eqref{eq.eq-t})
\be{eq.eq-t+}\begin{cases}\dvg (\e_\ell \nabla \theta_{\ell p})=g_{\ell p} \\[2mm]\int_{\ell\T}\theta_{\ell p}(y)\ dy=0.\end{cases}\ee
Let us determine the scalings of these contributions.
We claim that $N^{h p}_{\lambda\ell}$ is at most of order $\lambda$ whereas $P^{h p}_{\lambda\ell}$ is at most of order $\lambda^2 \ell^N$.

We start with $N^{h p}_{\lambda\ell}$ and recall  the following properties
on $w_\ell$ and $W_\ell$
$$
\nabla{w_{\ell k}}(y)=e_k+\nabla \chi_{\ell k}, \int_{\ell \T} |\nabla \chi_{\ell k}|^2 \le C,
\nabla W_{\ell ij}(y)=e_i \otimes e_j+\nabla X_{\ell ij}, \int_{\ell \T}|\nabla X_{\ell ij}|^2 \le C,
$$
which follow from energy estimates (note that the bound is uniform with respect to $\ell$).
Likewise, we have $\int_{\ell \T} |\nabla \theta_{\ell p}|^2 \le C$. This allows us to split the contribution in $(N^{h p}_{\lambda\ell})_{ijk}$ into four parts:
\begin{eqnarray*}
(N^{h p}_{\lambda\ell})_{ijk}&=&\ds  \lambda \ell^{N}\fint_{\ell\T} M_\ell(y)(e_k \otimes \nabla\theta_{\ell p}(y))\cdot e_i \otimes e_j\ dy
\\
&&+\ds  \lambda \ell^{N}\fint_{\ell\T} M_\ell(y)(\nabla{\chi_{\ell k}}(y)\otimes \nabla\theta_{\ell p}(y))\cdot e_i \otimes e_j \ dy
\\
&&+\ds  \lambda \ell^{N}\fint_{\ell\T} M_\ell(y)(e_k \otimes \nabla\theta_{\ell p}(y))\cdot\nabla X_{\ell ij}(y)\ dy
\\
&&+\ds  \lambda \ell^{N}\fint_{\ell\T} M_\ell(y)(\nabla{\chi_{\ell k}}(y)\otimes \nabla\theta_{\ell p}(y))\cdot\nabla X_{\ell ij}(y)\ dy.
\end{eqnarray*}
We treat the second and third terms alike using Cauchy-Schwarz' inequality to the effect that
\begin{multline*}
\Big|\int_{\ell\T} M_\ell(y)\Big( (\nabla{\chi_{\ell k}}(y)\otimes \nabla\theta_{\ell p}(y))\cdot e_i \otimes e_j +(e_k \otimes \nabla\theta_{\ell p}(y))\cdot\nabla X_{\ell ij}(y)\Big)\ dy\Big|
\\
\le C \Big(\int_{\ell\T} |\nabla \theta_{\ell p}|^2 )\Big)^\frac12
\Big(\int_{\ell\T} |\nabla X_{\ell ij}|^2 + |\nabla \chi_{\ell k}|^2\Big)^\frac12 \le C.
\end{multline*}
For the first term we use that $\nabla \theta_{\ell p}$ integrate to zero on $\ell \T$ by periodicity
and the specific form $M_\ell=M_1+(M_2-M_1)\chi_{B}$, followed by Cauchy-Schwarz' inequality, so that
\begin{eqnarray*}
 \Big|\int_{\ell\T} M_\ell(y)(e_k \otimes \nabla\theta_{\ell p}(y))\cdot e_i \otimes e_j\ dy \Big|&=&\Big|\int_{B} M_2 (e_k \otimes \nabla\theta_{\ell p}(y))\cdot e_i \otimes e_j \ dy\Big|
 \\
 &\le& C\Big(\int_{B} | \nabla\theta_{\ell p}|^2\Big) \le C.
\end{eqnarray*}
For the fourth term we need to use  more information on $\nabla\theta_{\ell p}$. By \eqref{eq.linfireg} and the analogue of \eqref{eq.int-ss}, \eqref{eq.diel-res}, we have for all $y \in \ell \T$
$$
|\nabla \theta_{\ell p}(y)| \le C(1+\int_{\ell \T} |\nabla \theta_{\ell p}|^2\Big)^\frac12 \le C.
$$
Hence, using $|ab| \le \frac12 (a^2+b^2)$,
$$
\Big|\int_{\ell\T} M_\ell(y)(\nabla{\chi_{\ell k}}(y)\otimes \nabla\theta_{\ell p}(y))\cdot\nabla X_{\ell ij}(y)\ dy\Big| \le C \int_{\ell\T} |\nabla{\chi_{\ell k}}|^2+|\nabla X_{\ell ij}|^2 \le C.
$$
We have thus proved that $|N^{h p}_{\lambda\ell}|\le C \lambda$.
{The argument to control $P^{h p}_{\lambda \ell}$ is similar and we  obtain $|P^{h p}_{\lambda \ell}|\le C \lambda^2 \ell^N$.}

\medskip

It remains to check that $P^{h p}_{\lambda \ell}$ is indeed of order $\lambda^2 \ell^N$.
To this aim, it is enough to replace correctors by their explicit approximation using the single-inclusion problem on the whole space  which we denote by $X_{\infty ij}$ and $\theta_{\infty p}$. We then define
$$
 P^{hp}_{\infty ij}\,:=\,\ds  \int_{\R^N} M_\infty(y)(\nabla{\theta_{\infty p}}(y)\otimes\nabla{\theta_{\infty p}}(y))\cdot\nabla W_{\infty ij}(y)\ dy
$$
On the one hand, a direct calculation using explicit formulas (see \cite[Section~17.2.1]{Torquato-02}) shows that $P^{h p}_{\infty}$ is of order
$1$ and that $|(\nabla{\theta_{\infty p}}(y)\otimes\nabla{\theta_{\infty p}}(y))\cdot\nabla W_{\infty ij}(y)| \le C (1+|y|)^{-2N}$.
On the other hand, by \cite{Pertinand}, one has
$$
\| \nabla (X_{\ell ij}-X_{\infty ij})\|_{L^2(\ell Y)}+\| \nabla (\theta_{\ell p}-\theta_{\infty p})\|_{L^2(\ell Y)} \le C \ell^{-\frac N2}
$$
so that $|P^{h p}_{\lambda\ell}-\lambda^2 \ell^N P^{h p}_{\infty}|\le C \lambda^2 \ell^\frac N2$,
and therefore $\frac1C \lambda^2 \ell^N \le \|P^{h p}_{\lambda\ell}\| \le C \lambda^2 \ell^N$.

\medskip

As we did above for $\ph$ when $\lambda,\ell^N\gg 1$, one can identify the leading order contribution $\tilde u$ to the solution $u$  of
\eqref{eq.hom-el} in the sense that $u =  \lambda^2 \ell^N  \tilde u+O(\lambda^2+\lambda \ell^N)$ (as an identity in $H^1(\Omega)$), where $\tilde u$ solves
\be{eq.hom-el+}
\ba{l} \dvg (L^h_\ell \nabla \tilde u)=-  \dvg   \bar P^p_\ell f_p^2 \\[2mm]
\tilde u=0 \mbox{ on }\partial\Om,\ea\ee
where $ (\bar P^p_{\ell})_{ij}=: (\lambda^2 \ell^N)^{-1} (P^{h p}_{\lambda \ell})_{ij}$ is of order 1.
In particular, this yields enhancement of elastostriction by a factor $\lambda^2 \ell^N$.

\appendix

\section{Extension to the random setting}

For simplicity, we only consider the dielectric problem (the coupling to elasticity is straightforward once
the needed regularity results are proved) and we place ourselves in a setting similar to the example of Section~\ref{sec.ac} with spherical inclusions.
Let $\Pc=\{x_n, n \in \N\}$ be a stationary ergodic point process on $\R^N$ such that for all $n\ne m$, $|x_n-x_m|\ge 2+\delta$ for some deterministic $\delta>0$, and ys" (we denote by $\expec{\cdot}$ the associated expectation), and $\I=\cup_{x \in \Pc}B(x)$ denote the (random) set of inclusions in $\R^N$.
We then define
$$
\e:x\mapsto \Id+(\bar \e-1)\Id\chi_{\I}.
$$
In this appendix, $C$ denotes a constant  that may change from line to line,
depends on $N$, $\Omega$, $\bar \e$, controlled norms of $\phi$ and $f$, and the law of $\Pc$, but which is independent of $g$ and $\psi$ (see below), unless otherwise explicitly stated (using subscripts).

\subsection{Definition of $g$ and $\psi$}

We start with the definition of $g$ and $\psi$, cf.~\eqref{eq.psi} in the periodic setting.
For conciseness, for all $x\in \R^N$ we denote by $\B(x)$ the ball $B_{1+\delta/2}(x)$.
In particular, by definition, the balls $\{\B(x)\}_{x \in \Pc}$.
We assume that $g: \R^N \to \R^N$ is a stationary random field
supported on $\cup_{x \in \Pc} \B(x)$ and satisfying $\int_{\B(x)} g=0$ for all $x\in \Pc$
and $\|g \|_{L^\infty(\R^N)} <\infty$.
Under this specific assumption, there exists a stationary field $\nabla \psi \in L^2_\loc(\R^N)$ with vanishing expectation $\expec{\nabla \psi}=0$ and finite second moment $\expec{|\nabla \psi|^2} <\infty$ satisfying
almost surely
\be{eq.psi-rand}  \triangle \psi(y)=  g(y).\ee
Let us give the short argument in favor of the well-posedness of \eqref{eq.psi-rand} for completeness. As customary in the field, we first add a massive regularization of order $T\gg 1$ and consider the equation on $\R^N$
\begin{equation}\label{approxT}
\frac1T \psi_T(y)- \triangle \psi_T(y)=  -g(y),
\end{equation}
which is well-posed in the space $H^1_{\uloc}(\R^N)=\{ \zeta \in H^1_\loc(\R^N)\,|\, \sup_{x\in \R^N} \int_{B(x)} (\zeta^2+|\nabla \zeta|^2)<\infty\}$. The argument is standard (see e.g.~\cite[Lemma~2.7]{Gloria-Otto-10b}): we first solve the equation on balls $B_R$ with homogeneous Dirichlet boundary conditions, and pass to the limit $R \nearrow+\infty$
using a uniform a priori bound in $H^1_{\uloc}(\R^N)$. This bound relies on the Caccioppoli inequality, which we presently work out in our setting. We display the argument in the whole space, assuming that all the quantities that appear are finite. The argument is the same for the approximations on the balls $B_R$ (for which all quantities are indeed finite). 
Let $\eta_T:x\mapsto \exp(-c{|x|}/{\sqrt{T}})$, and test the equation with $\eta_T^2 \psi_T$.
This yields after integration by parts and rearranging the terms
$$
\int_{\R^N} \frac 1T \psi_T^2 \eta_T^2 + \int_{\R^N} \eta_T^2 |\nabla \psi_T|^2 = -\int_{\R^N} g \eta_T^2 \psi_T -2\int_{\R^N} \eta_T \psi_T \nabla \eta_T\cdot \nabla \psi_T.
$$
The second term is standard: since $|\nabla \eta_T|\le \frac{c}{\sqrt T}\; \eta_T$, we have for $c$ small enough,
$$
\Big|2\int_{\R^N} \eta_T \psi_T \nabla \eta_T\cdot \nabla \psi_T\Big| \le \frac14 \int_{\R^N} \eta_T^2 |\nabla \psi_T|^2 + 4\int_{\R^N} |\nabla \eta_T|^2 \psi_T^2 \le  \frac14 \int_{\R^N} \eta_T^2 |\nabla \psi_T|^2 + \frac14 \int_{\R^N} \frac1T  \psi_T^2\eta_T^2.
$$
We then use the specific properties of $g$ to reformulate the first term as
$$
\int_{\R^N} g \eta_T^2 \psi_T = \sum_{n\in \N} \int_{\B(x_n)} g (\eta_T^2\psi_T-\gamma_n)
$$
where the $\gamma_n$'s are arbitrary constants (since $g$ has vanishing average on the
$\B(x_n)$'s). By Poincar\'e-Wirtinger's inequality on the $\B(x_n)$'s, we thus have
$$
\Big|\int_{\B(x_n)} g(y) (\eta_T^2\psi_T-\gamma_n)\Big| \,\le \, C \Big(\int_{\B(x_n)} g^2\Big)^\frac12 \Big(\int_{\B(x_n)} |\nabla (\eta_T^2 \psi_T)|^2\Big)^\frac12
$$
and we expand the second factor as, using again that $|\nabla \eta_T| \le \frac c{\sqrt{T}} \eta_T$,
$$
\int_{\B(x_n)} |\nabla (\eta_T^2 \psi_T)|^2\,\le \, C \int_{\B(x_n)} \eta_T^4 |\nabla \psi_T|^2+\eta_T^2 |\nabla \eta_T|^2 \psi_T^2
\le \Big( \sup_{\B(x_n)} \eta_T^2 \Big)
\int_{\B(x_n)} \eta_T^2 |\nabla \psi_T|^2+\eta_T^2 \frac{c^2}T \psi_T^2.
$$
By the inequality $ab \le \frac12(\frac1{C^2}a^2+C^2b^2)$ for an appropriate constant $C$, this yields
$$
\Big|\int_{\B(x_n)} g(y) (\eta_T^2\psi_T-\gamma_n)\Big| \,\le \, C \sup_{\B(x_n)} \eta_T^2  \int_{\B(x_n)} g^2+ \frac14\int_{\B(x_n)} \frac1T \eta_T^2\psi_T^2+\eta_T^2 |\nabla \psi_T|^2.
$$
Altogether, these estimates combine to the a priori estimate
\begin{equation}\label{e.caccio}
\int_{\R^N} \frac 1T \psi_T^2 \eta_T^2 + \int_{\R^N} \eta_T^2 |\nabla \psi_T|^2
\le C \|g\|_{L^\infty(\R^N)}^2 \sum_{n\in \N} \sup_{\B(x_n)} \eta_T^2.
\end{equation}
From this, it is now standard to deduce that there exists a unique random field $\psi$
such that $\nabla \psi$ is stationary and has finite second moment 
\begin{equation}\label{e.second-moment}
\expec{|\nabla \psi|^2}^\frac12 \le C \expec{\chi_\I (0)} \|g\|_{L^\infty(\R^N)},
\end{equation}
and $\psi$
solves \eqref{eq.psi-rand} almost surely in the distributional sense, cf.~\cite{PapaVara}.

\medskip

We conclude with a quick argument in favor of the additional a priori bound
\begin{equation}\label{e.nabla2}
\expec{|\nabla^2 \psi|^2}\le \expec{g^2}.
\end{equation}
Starting point is \eqref{approxT} which is satisfied almost surely on $\R^N$ and implies that $\nabla^2 \psi_T \in L^2_\loc(\R^N)$. Again, it is enough to prove an a priori estimate on $\nabla^2 \psi_T$ and pass to the limit $T\nearrow +\infty$. 
Up to proceeding at the level of approximations on balls $B_R$, we may assume that all the quantities involved below are finite.
We then test  \eqref{approxT}  with $-\eta^2_T \triangle \psi_T$.
Treating the massive term $\frac1T \eta_T^2 \psi_T \triangle \psi_T$ as above, and using
\eqref{e.caccio}, this 
yields
$$
\int_{\R^N}\eta_T^2 (\triangle \psi_T)^2 \, \le \, \int_{\R^N} \eta^2_T |g||\triangle \psi_T|+\frac CT \|g\|_{L^\infty(\R^N)}^2 \sum_{n\in \N} \sup_{\B(x_n)} \eta_T^2,
$$
which (by controlling the roots of a trinomial) implies 
$$
\int_{\R^N}\eta_T^2 (\triangle \psi_T)^2 \le   \int_{\R^N} \eta^2_T g^2 +
\frac{2C}T\|g\|_{L^\infty(\R^N)}^2 \sum_{n\in \N} \sup_{\B(x_n)} \eta_T^2.
$$
It remains to reformulate the left-hand side to recognize the the Hessian. 
After two integrations by parts, we have
$$
\int_{\R^N}\eta_T^2 (\triangle \psi_T)^2 = \int_{\R^N}\eta_T^2 |\nabla^2 \psi_T|^2
+ 2\sum_{ij} \int_{\R^N}\eta_T \partial_j \eta_T \partial_j \psi_T \partial_{ij}^2 \psi_T
- 2 \sum_{i} \eta_T \partial_i \eta_T \partial_i \psi_T \triangle \psi_T,
$$
which we rewrite, using that $ |\nabla \eta_T|\le \frac{c}{\sqrt{T}} \eta_T$ for $c$ small enough, as 
$$
 \int_{\R^N}\eta_T^2 |\nabla^2 \psi_T|^2 \le (1+\frac CT) \int_{\R^N}\eta_T^2 (\triangle \psi_T)^2
 +\frac CT \int_{\R^N}\eta_T^2 |\nabla \psi_T|^2
$$
for some $C$ depending only on $N$ and $c$.
We have thus proved 
$$
 \int_{\R^N}\eta_T^2 |\nabla^2 \psi_T|^2 \le  (1+\frac CT) \int_{\R^N} \eta^2_T g^2+\frac CT \|g\|_{L^\infty(\R^N)}^2 \sum_{n\in \N} \sup_{\B(x_n)} \eta_T^2.
$$
Taking the expectation and letting $T\nearrow +\infty$ yields the claim \eqref{e.nabla2}.

\subsection{Qualitative homogenization of the dielectrics}

Once $g$ and $\psi$ are defined as above, the proof of the qualitative homogenization of the dielectrics follows the proof of Section~\ref{sec.ch}, replacing periodicity by stationarity -- the adaptation is standard and left to the reader (see e.g.~\cite[Chapter~7]{JKO94} or \cite{PapaVara}).

\medskip

In particular, \eqref{eq.hom-diel} is replaced by
$
\e^h\vec{e}_j:=\expec{\e \nabla w_j},
$
and the formula \eqref{eq.defaj} for $a$ takes the form
$
a_j:=\expec{ \nabla \psi \cdot\nabla w_j}.
$

\subsection{Improved integrability in homogenization of the dielectrics}

In this paragraph we extend the results of Section~\ref{sec.ic} to the random setting.
Since the results are based on large-scale regularity for random elliptic operators, we need to make some quantitative mixing assumptions on the point process $\Pc$. In particular,
a hardcore Poisson point process or the random parking measure will do, cf.~\cite{DG2,GNO-reg}.

\medskip

We start with the extension of Proposition~\ref{prop.Lqreg}.
\begin{proposition}
For all $2\le q,q' <\infty$, the sequence $\nabla\phd$ satisfies
$\expec{\|\nabla \phd\|_{L^q(\Om;\R^N)}^{qq'}} \le C_{q,q',g}$, independently of $\delta$.
In particular, $\nabla \phd$ is bounded in $L^q(\Om;\R^N)$ along any subsequences of $\delta$
almost surely.
%
\end{proposition}
\begin{proof}

\noindent{\sf Step 1.} First, we apply the large-scale Calder\'on-Zygmund estimates of \cite[Theorem~7.7]{AKM-book}   to \eqref{eps2}.
In this random setting, the estimate involves a stationary random field $r_*\ge 1$ on $\R^N$, which, in the examples considered above, satisfies $\expec{\exp(\frac1C r_*(0))}\le 2$
for some finite constant $C$, cf.~\cite[Theorem~4]{GNO-reg} (by stationarity, this holds
for $r_*(0)$ replaced by $r_*(x)$ for all $x\in \R^N$).
We then obtain for all $q\ge 2$, with the short-hand notation $B_{*,\delta}(x):=B_{\delta r_*(\frac x \delta)}(x)$
\begin{multline*}
\int_{\Om} \Big(\fint_{B_{*,\delta}(x)} \chi_\Om(z) |\nabla (\ph^\delta-\phi)|^2(z)\ dz\Big)^\frac q2\ dx \\[2mm] \le \, C_q  \Big( \| \nabla (\ph^\delta-\phi) \|_{L^2(\Om)}^q+ \|f\td-\ed\nabla \phi+\nabla \zd\|_{L^q(\Omega)}^q\Big).
\end{multline*}
As in the periodic setting, since $\zd$ is the unique $H^1_0(\Om)$-solution to
$\triangle \zd=\nabla f\cdot\td$, we have by maximal regularity for the Laplacian
$$
\|f\td-\ed\nabla \phi+\nabla \zd\|_{L^q(\Omega)}^q \,\le \, C \Big((\|f\|_{C^{1,\alpha}(\Omega)}+1)
\|\td\|_{L^q(\Omega)}+\|\phi\|_{C^{1,\alpha}(\Omega)}\Big)^q
$$
and we have to control $\|\td\|_{L^q(\Omega)}$. As opposed to the periodic setting, this is a random quantity.
We proceed in two steps. First, taking the derivative of \eqref{eq.psi-rand}
and using deterministic Calder\'on-Zygmund estimates for the Laplacian,
we have
for all $R\ge 1$ %
$$
\fint_{B_R} |\nabla^2 \psi|^{qq'}  \, \le \, C_{q,q'}\Big(
\Big(\fint_{B_{2R}} |\nabla^2 \psi|^{2} \Big)^{qq'/2}+\fint_{B_R} |g|^{qq'} \Big)
$$
so that by taking the limit $R\nearrow +\infty$ we obtain by stationarity of these random fields and the ergodic theorem
$$
\expec{ |\nabla^2 \psi|^{qq'}}  \, \le \, C_{q,q'}
\Big(\expec{|\nabla^2 \psi|^{2}}^{qq'/2}+\expec{ |g|^{qq'} }\Big).
$$
Using with \eqref{e.nabla2} and H\"older's inequality in probability, this turns into $\expec{ |\nabla^2 \psi|^{qq'}}  \, \le \, C_{q,q'} \expec{ |g|^{qq'} }$.
Combined with \eqref{e.second-moment} and Poincar\'e's inequality in form of
$
\|\nabla \psi \|_{L^{qq'}(B)} \, \le \, C \Big( \|\nabla^2 \psi \|_{L^{qq'}(B_2)}
+ \|\nabla \psi \|_{L^{2}(B)}\Big),
$
and using the ergodic theorem as above, this yields
\begin{equation} \label{estoc}
\expec{ |\nabla \psi|^{qq'}|}  \, \le \, C_{q,q'}  \|g\|_{L^\infty(\R^N)}^{qq'},
\end{equation}
that is, the desired control
$
\expec{\|\td\|_{L^{q}(\Omega)}^{qq'}} \, \le\, C_{q,q'}  \|g\|_{L^\infty(\R^N)}^{qq'}.
$
As in the periodic setting, we also have
\[
\int_{\Omega} \Big(\fint_{B_{*,\delta}(x)} \chi_\Omega |\nabla \phi|^2\ dz\Big)^\frac q2\,dx\le\,
 \int_{\R^N} \fint_{B_{*,\delta}(x)} \chi_\Omega |\nabla \phi|^q
\,dzdx \le \, C \int_{\R^N} \chi_\Omega |\nabla \phi|^qdx = \int_{\Omega}   |\nabla \phi|^q\ dx,
\]
where we used \cite[(140)]{GNO-reg} in form of $\int_{\R^N} \fint_{B_{*,\delta}} \sim \int_{\R^N}$.
Hence, we have proved that
\be{eq.estq2sto}
\Expec{\Big(\int_\Om\Big(\fint_{B_{*,\delta}(x)} \chi_\Omega |\nabla \ph^\delta|^2\ dz\Big)^\frac q2\,dx\Big)^{q'}} \le C_{q,q',g}
\ee
for some finite constant $C_{q,q'}$ depending on $q,q'$ and on $\|f\|_{C^{1,\alpha}(\Om)}$, $\|\phi\|_{C^{1,\alpha}(\Om)}$, and $\|g\|_{L^\infty}$.
As opposed to the periodic setting, we have to reformulate this estimate in order to remove the stochastic dependence of the local averages upon the random field $r_*$.
The rest of this step is dedicated to the proof of
\be{eq.estq2sto+}
\Expec{\Big(\int_\Om\Big(\fint_{B_{\delta}(x)} \chi_\Omega |\nabla \ph^\delta|^2\ dz\Big)^\frac q2\,dx\Big)^{q'}} \le C_{q,q',\Omega,g}.
\ee
As customary in the field, this can be done at the price of some (arbitrarily small) loss of stochastic integrability (the dependence of the constants in \eqref{eq.estq2sto} and \eqref{eq.estq2sto+} with respect to $q$ and $q'$ are different). Since we are not interested in the precise stochastic integrability in this contribution, we display an elementary (and suboptimal) proof of this improvement of \eqref{eq.estq2sto}.
The argument relies on the estimate
%
\begin{multline}\label{e.virer*}
\int_\Om \Big(\fint_{B_{\delta}(x)} \chi_\Om |\nabla \ph^\delta|^2\ dz\Big)^\frac q2\,dx
\\
\le \, C\big(\frac{\delta\inf_{\Om}  r_*(\tfrac \cdot\delta)^N}{|\Om|}+1\big) \int_{\Om}r_*(\tfrac x \delta)^{N(q-2)}\Big(\fint_{B_{*,\delta}(x)} \chi_\Om |\nabla \ph^\delta|^2\ dz\Big)^\frac q2\,dx
\end{multline}
in favor of which we presently argue (note that the averages on the left-hand side are made on balls of fixed radius $\delta$).
Following  \cite{GNO-reg}, we replace the integral of local averages by a sum on a partition.   In particular, by \cite[(139)]{GNO-reg}, there exists a partition of $\R^N$ into a family of cubes  $\calQ:=\{Q\}_Q$ such that $\sup_Q r_* \le C \inf_Q r_*$, $\diam(Q) \sim \inf_Q r_*$ and for all functions $h\ge 0$ and exponents $\gamma\ge 1$ we  have $\int_{\R^N} \big(\fint_{B_*(x)}h\big)^\gamma \sim \sum_Q |Q| (\fint_Q h)^\gamma$.
Denote by $\calQ_\delta(\Om)$ the smallest subset of $\calQ$ which contains $\frac 1\delta \Om$ in the sense that $\frac1\delta \Om \subset \cup_{Q \in \calQ_\delta(\Om)}Q$.
For convenience, we call $Q_\delta(x)$ the cube of radius $\delta$ centered at $x\in \R^N$.
Then we have by the discrete $\ell^1-\ell^\gamma$ estimate
\begin{eqnarray*}
\int_\Om \Big(\fint_{B_{\delta}(x)} h\chi_\Om\Big)^\gamma dx
&\le & C \sum_{Q \in \calQ_\delta(\Om)} \int_{\delta Q}  \Big(\fint_{Q_{\delta}(x)} h\chi_\Om\Big)^\gamma  dx
\\
&\le & 3^NC  \sum_{Q \in \calQ_\delta(\Om)} \diam(Q)^{(\gamma-1) N} \int_{\delta Q}  \Big(\fint_{\delta Q} h\chi_\Om\Big)^\gamma  dx.
\end{eqnarray*}
We now distinguish two cases: If $\sup_{Q \subset \calQ_\delta(\Om)} \diam(Q) \le \tfrac1\delta \diam(\Omega)$, then  \cite[(146)]{GNO-reg} combined with the property $\sup_Q r_* \le C \inf_Q r_*$ yields
$$
\int_\Om \Big(\fint_{B_{\delta}(x)} h\chi_\Om\Big)^\gamma dx \,\le \, C
\int_{\Om} r_*(\tfrac x\delta)^{(\gamma-1) N} \Big(\fint_{B_{*,\delta}(x)}h(z)\chi_\Om(z)\ dz\Big)^\gamma dx,
$$
whereas if $\sup_{Q \subset \calQ_\delta(\Om)} \diam(Q) > \tfrac1\delta \diam(\Omega)$,
then
$$
\int_\Om \Big(\fint_{B_{\delta}(x)} h\chi_\Om\Big)^\gamma dx \,\le \,
C\big(\frac{\delta\inf_{\Om}  r_*(\tfrac \cdot\delta\big)^N}{|\Om|}+1)\int_{\Om} r_*(\tfrac x\delta)^{(\gamma-1) N} \Big(\fint_{B_{*,\delta}(x)}h(z)\chi_\Om(z)\ dz\Big)^\gamma dx.
$$
Applied to $h =|\nabla \ph^\delta|^2$ and $\gamma=\frac q2$ this proves \eqref{e.virer*}.
We conclude by deriving \eqref{eq.estq2sto+} from \eqref{e.virer*}.
To that end, we control the infimum of $r_*$ on $\frac1\delta \Om$ by
its average, use several times H\"older's inequality in probability,
the moment bound $\expec{\exp(\frac1C r_*)} \le 2$, and  the triangle inequality in probability in the form $\expec{\big(\int_{\Omega} |h|)^{q'}}^\frac1{q'}
\le \int_{\Omega} \expec{|h|^{q'}}^\frac1{q'}$. This yields
\begin{eqnarray*}
&&\Expec{\Big(\int_\Om \Big(\fint_{B_{\delta}(x)} \chi_\Om |\nabla \ph^\delta|^2\Big)^\frac q2\,dx\Big)^{q'}}^\frac1{q'}
\\
&\le& C \Expec{\Big(1+\delta \fint_\Om r_*(\tfrac \cdot \delta)\Big)^{2q'}}^\frac1{2q'}\Expec{\Big(\int_\Om r_*(\tfrac x\delta)^{\frac{q-2}{2} N} \Big(\fint_{B_{*,\delta}(x)} \chi_\Om |\nabla \ph^\delta|^2\Big)^\frac q2\,dx\Big)^{2q'}}^\frac1{2q'}
\\
&\le &C_{q,q'}\int_{\Om}\Expec{r_*(\tfrac x \delta)^{Nq'(q-2)}\Big(\fint_{B_{*,\delta}(x)} \chi_\Om |\nabla \ph^\delta|^2\ dz\Big)^{qq'}}^\frac1{2q'} dx
\\
&\le &C_{q,q'}\int_{\Om}\Expec{r_*(\tfrac x \delta)^{2Nq'(q-2)}}^\frac1{4q'}\Expec{\Big(\fint_{B_{*,\delta}(x)} \chi_\Om |\nabla \ph^\delta|^2\ dz\Big)^{2qq'}}^\frac1{4q'} dx
\\
&\le &C_{q,q',\Omega}\Expec{\int_{\Om}\Big(\fint_{B_{*,\delta}(x)} \chi_\Om |\nabla \ph^\delta|^2\ dz\Big)^{2qq'} dx}^\frac1{4q'}.
\end{eqnarray*}
Combined with \eqref{eq.estq2sto}, this entails \eqref{eq.estq2sto+}, which  enables us to control $\nabla \phd$ on scales larger than $\delta$. In a second step, we will derive an estimate for small scales, that is for scales smaller than $\delta$.

\medskip
\noindent{\sf Step 2.} Take a point $x\in \Om$ and consider the cube $Q_{2\delta}(x)$ of side-length $2\delta$ centered at $x$. We blow up equation \eqref{eq.diel2} so as to obtain an equation on {$Q_2(0)$}.

To that effect, we set
$\Phi^\delta(z):=\frac1\delta\phd(x+\delta z)$, which satisfies
\be{eq.diel-res-st}\dvg_z (\e(\frac x\delta+z) \nabla_z\Phi^\delta-\!f(x+\delta z)\td(x+\delta z))\!= -\delta\td(x+\delta z)\!\cdot\!\nabla f(x+\delta z) \mbox{ in } (\Om-\{x\})/\delta\cap {Q_2(0)}.\ee
The only difference with the periodic setting is that $\td$ is now random.
By elliptic regularity for the Laplacian in \eqref{eq.psi-rand}, we have
$$
\| \nabla \psi \|_{C^{1,\alpha}(B(x))} \,\le \, C \Big(\|\nabla \psi \|_{H^1(B_2(x))}+\|g \|_{C^{0,\alpha}(B_2(x))}\Big).
$$
Using this bound, the same argument as in the periodic setting allows to conclude that
\be{eq.est-localsto}
 \sup_{B_{\delta/d}(x)\cap \Om}|\nabla \phd|\le C \Big(\Big(\fint_{\Om\cap B_\delta(x)}|\nabla\phd(y)|^2\ dy\Big)^\frac12+C'+\|\nabla \psi \|_{H^1(B_2(x/\delta))}+\|g \|_{C^{0,\alpha}(B_2(x/\delta))}\Big),\ee
where  $C'$ is a constant that only depends on $\|f\|_{C^{1,\alpha}(\Om)}$ and $\|\phi\|_{C^{1,\alpha}(\Om)}$.

\medskip

\noindent{\sf Step 3.} We combine the estimates obtained in the first two steps as in the
periodic setting. This yields
\begin{multline*}
\Expec{\Big(\int_\Om  |\nabla \ph^\delta|^q\Big)^{q'}} \le C\Expec{\Big(\int_\Om\Big(\fint_{B_{\delta}(x)} \chi_\Omega |\nabla \ph^\delta|^2\ dz\Big)^\frac q2\,dx\Big)^{q'}}
\\
+ \Expec{\Big(\int_\Om \|\nabla \psi \|_{H^1(B_2(x/\delta))}^q+\|g \|_{C^{0,\alpha}(B_2(x/\delta))}^q\Big)^{q'}},
\end{multline*}
which, by \eqref{eq.estq2sto+}, stationarity of $\nabla \psi$ and $g$ and by \eqref{estoc}, entails
$\Expec{\Big(\int_\Om  |\nabla \ph^\delta|^q\Big)^{q'}} \le C_{q,q'}$, as claimed.
\end{proof}

\subsection{Enhancement of the dielectric coefficient}

The successful strategy we used in the periodic setting to prove the enhancement of the dielectric coefficient {in the dilute case} can be implemented in the random setting considered here, using the same $g$ as in the periodic setting around the spherical inclusions. The proof raises additional technicalities, which can all be dealt with as we did above for the other results.
The analysis is again inspired by recent results on the Clausius-Mossotti formula. As opposed to the periodic setting, there are many ways to thin a random point process and reach the dilute regime. As in the periodic setting, one may use geometric dilation and consider $\Pc_\ell= \ell \Pc$, but one may also attach a Bernoulli variable to each point and discard it if the variable is 0 -- this leads to thinning by random deletion.
For general thinning, we refer to the recent work  \cite{DG-20c} on the Einstein formula.
In the case of geometric dilation, \cite{Pertinand} provides tools which extend the results we used above in the periodic setting (this is however more involved since  massive regularization is needed and this generates an additional error term that can be controlled using
results of \cite{GNO-quant}). Likewise, for random deletion, one can combine tools introduced in \cite{DG-16a} with \cite{GNO-quant}.

\section*{Acknowledgements}

{The work of GAF and OLP was supported by the National Science Foundation through the Grant DMREF-1921969.} AG has received funding from the European Union's Horizon 2020 research and innovation programme under grant agreement No 864066.

\end{document}